\documentclass[reqno, 12pt]{amsart}
\pdfoutput=1
\makeatletter
\let\origsection=\section \def\section{\@ifstar{\origsection*}{\mysection}}
\def\mysection{\@startsection{section}{1}\z@{.7\linespacing\@plus\linespacing}{.5\linespacing}{\normalfont\scshape\centering\S}}
\makeatother

\usepackage{amsmath,amssymb,amsthm}
\usepackage{mathrsfs}
\usepackage{mathabx}\changenotsign
\usepackage{dsfont}
\usepackage[dvipsnames]{xcolor}

\usepackage{xcolor}
\usepackage[backref]{hyperref}
\hypersetup{
	colorlinks,
	linkcolor={red!60!black},
	citecolor={green!60!black},
	urlcolor={blue!60!black}
}

\usepackage[dvipsnames]{xcolor}
\usepackage[open,openlevel=2,atend]{bookmark}

\usepackage[abbrev,msc-links,backrefs]{amsrefs}
\usepackage{doi}

\renewcommand{\PrintDOI}[1]{\doi{#1}}

\usepackage[T1]{fontenc}
\usepackage{lmodern}
\usepackage[babel]{microtype}
\usepackage[english]{babel}

\usepackage{geometry}
\numberwithin{equation}{section}
\numberwithin{figure}{section}

\usepackage{enumitem}
\def\rmlabel{\upshape({\itshape \roman*\,})}

\def\alabel{\upshape({\itshape \alph*\,})}

\def\ag#1{	\tikz{\def\nn{#1};
		\pgfmathsetmacro\wie{3*\nn-1};
		\pgfmathsetmacro\mm{5*\nn};
		\pgfmathsetmacro\kk{\nn-1};
		
		\foreach \i in {0,...,\mm}{
			\coordinate (x\i) at (\i*360/\wie:1.3cm);}
		
		\foreach \i in {0,...,\wie}{	
			\foreach \j [evaluate=\j as \k using \i+\j+\nn] in {0,...,\kk}{
				\draw[green!75!black] (x\i)--(x\k);}
		}		
		
		\foreach \i in {0,...,\wie}{	
			\draw[black, fill=black] (x\i) circle (1pt);
		}		}}

		\newcommand{\mref}[1]{\ifmmode\textrm{\ref{#1}}\else\ref{#1}\fi}

		\let\polishlcross=\l
		\def\l{\ifmmode\ell\else\polishlcross\fi}

		\let\sm=\smallsetminus

		\makeatletter
		\def\moverlay{\mathpalette\mov@rlay}
		\def\mov@rlay#1#2{\leavevmode\vtop{   \baselineskip\z@skip \lineskiplimit-\maxdimen
				\ialign{\hfil$\m@th#1##$\hfil\cr#2\crcr}}}
		\newcommand{\charfusion}[3][\mathord]{
			#1{\ifx#1\mathop\vphantom{#2}\fi
				\mathpalette\mov@rlay{#2\cr#3}
			}
			\ifx#1\mathop\expandafter\displaylimits\fi}
		\makeatother
		
		\newcommand{\dcup}{\charfusion[\mathbin]{\cup}{\cdot}}

		\DeclareFontFamily{U}  {MnSymbolC}{}
		\DeclareSymbolFont{MnSyC}         {U}  {MnSymbolC}{m}{n}
		\DeclareFontShape{U}{MnSymbolC}{m}{n}{
			<-6>  MnSymbolC5
			<6-7>  MnSymbolC6
			<7-8>  MnSymbolC7
			<8-9>  MnSymbolC8
			<9-10> MnSymbolC9
			<10-12> MnSymbolC10
			<12->   MnSymbolC12}{}
		\DeclareMathSymbol{\powerset}{\mathord}{MnSyC}{180}
		
		\usepackage{tikz}
		\usetikzlibrary{calc,decorations.pathmorphing}
		\usetikzlibrary{math}
		\pgfdeclarelayer{background}
		\pgfdeclarelayer{foreground}
		\pgfdeclarelayer{front}
		\pgfsetlayers{background,main,foreground,front}
		
		\usepackage{subcaption}

		\let\epsilon=\varepsilon
		\let\eps=\epsilon
		\let\rho=\varrho
		\let\theta=\vartheta

		\def\NN{{\mathds N}}
		\def\GG{{\mathds G}}
		\def\ZZ{{\mathds Z}}
		
		\def\RR{{\mathds R}}
		
		\def\Ind{{\mathds 1}}
		\def\o{\omega_{\mu\nu}}

		\newcommand{\cG}{\mathscr{G}}
		\newcommand{\cH}{\mathscr{H}}

		\newcommand{\ex}{\mathrm{ex}}

		\newcommand{\ccC}{\mathscr{C}}

		\newcommand{\al}{\alpha}
		
		\let\Nf=\ccN

		\newcommand{\vv}[2]{{\Upsilon}_{#1}^{#2}}
		\newcommand{\ve}{\vv{i}{\mu\nu}}
		\newcommand{\vl}{\Upsilon_i}

		\newtheoremstyle{note}  {4pt}  {4pt}  {\sl}  {}  {\bfseries}  {.}  {.5em}          {}
		\newtheoremstyle{introthms}  {3pt}  {3pt}  {\itshape}  {}  {\bfseries}  {.}  {.5em}          {\thmnote{#3}}
		\newtheoremstyle{remark}  {2pt}  {2pt}  {\rm}  {}  {\bfseries}  {.}  {.3em}          {}
		
		\theoremstyle{plain}
		\newtheorem{theorem}{Theorem}[section]
		\newtheorem{lemma}[theorem]{Lemma}

		\newtheorem{conj}[theorem]{Conjecture}
		\newtheorem{cor}[theorem]{Corollary}
		\newtheorem{fact}[theorem]{Fact}
		\newtheorem{claim}[theorem]{Claim}

					\newtheorem*{conj*}{Conjecture}

		\theoremstyle{note}		
		\newtheorem{dfn}[theorem]{Definition}

		\theoremstyle{remark}

		\usepackage{accents}

		\usepackage{mathtools}
		\usepackage{lineno}
		\newcommand*\patchAmsMathEnvironmentForLineno[1]{%
			\expandafter\let\csname old#1\expandafter\endcsname\csname #1\endcsname
			\expandafter\let\csname oldend#1\expandafter\endcsname\csname end#1\endcsname
			\renewenvironment{#1}%
			{\linenomath\csname old#1\endcsname}%
			{\csname oldend#1\endcsname\endlinenomath}}%
		\newcommand*\patchBothAmsMathEnvironmentsForLineno[1]{%
			\patchAmsMathEnvironmentForLineno{#1}%
			\patchAmsMathEnvironmentForLineno{#1*}}%
		\AtBeginDocument{%
			\patchBothAmsMathEnvironmentsForLineno{equation}%
			\patchBothAmsMathEnvironmentsForLineno{align}%
			\patchBothAmsMathEnvironmentsForLineno{flalign}%
			\patchBothAmsMathEnvironmentsForLineno{alignat}%
			\patchBothAmsMathEnvironmentsForLineno{gather}%
			\patchBothAmsMathEnvironmentsForLineno{multline}%
		}
		\usepackage{multicol}
		\usepackage{todonotes}

		\usepackage{marginnote}

		\newcommand{\G}{\Gamma}

\newcommand{\Nn}{N}
\let\lra=\longrightarrow

\begin{document}
\title[Andr\'asfai and Vega graphs in Ramsey-Tur\'an theory]{Andr\'asfai and Vega graphs 
in Ramsey-Tur\'an theory}
			
\author[T.~\L uczak]{Tomasz \L uczak}
\address{A. Mickiewicz University, Department of Discrete Mathematics, Pozna\'n, Poland}
\email{tomasz@amu.edu.pl}
\email{joaska@amu.edu.pl}
\thanks{The first author is partially supported by National Science Centre, Poland, 
grant 2017/27/B/ST1/00873.}
			
\author[J.~Polcyn]{Joanna Polcyn}
			
\author[Chr.~Reiher]{Christian Reiher}
\address{Fachbereich Mathematik, Universit\"at Hamburg, Hamburg, Germany}
\email{Christian.Reiher@uni-hamburg.de}
			
\subjclass[2010]{Primary: 05C35.}
\keywords{Ramsey-Tur\'an theory, extremal graph theory, triangle-free graphs.}
			
\begin{abstract}
	Given positive integers $n\ge s$, we let $\ex(n,s)$ denote the maximum number of 
	edges in a triangle-free graph $G$ on $n$ vertices with $\alpha(G)\le s$. 
	In the early sixties Andr\'{a}sfai conjectured that for  $n/3<s<n/2$ the function 
	$\ex(n, s)$ is piecewise quadratic with critical values at $s/n={k}/({3k-1})$
	for $k\in \NN$. We confirm that this is indeed the case whenever $s/n$ is slightly 
	larger than a  critical value, thus determining $\ex(n,s)$ 
	for all $n$ and $s$ such that $s/n\in [{k}/({3k-1}), {k}/({3k-1})+\gamma_k]$, 
	where $\gamma_k=\Theta(k^{-6})$.
\end{abstract}
			
\maketitle
			
\section{Introduction}

The structure of dense triangle-free graphs has been the subject of extensive studies 
for a long time. The first result in this direction is Mantel's celebrated theorem~\cite{M} 
from~1907, which states that balanced bipartite graphs are the densest 
triangle-free graphs. 
It is natural to ask for the densest triangle-free graphs when we impose some additional 
restrictions on them; for instance, we may bound their chromatic or independence number.

Let us first discuss the case when we require the chromatic number of a dense triangle-free 
graph to be large. It is easy to see that in this case the appropriate measure of the density 
of a graph is not the number of its edges (as one can always add a small graph of large 
chromatic number to a complete bipartite graph), but rather its minimum degree. 
This avenue of research was started by Andr\'asfai, Erd\H{o}s, and S\'os~\cite{AES} who 
showed that among triangle-free graphs with chromatic number three those with
the largest minimum degree are `balanced blow-ups' of the pentagon. Erd\H{o}s and 
Simonovits~\cite{ES} noticed that a construction due to Hajnal shows that for 
every $k\ge 2$, $\epsilon>0$, and sufficiently large~$n$, there exists a triangle-free graph 
on $n$ vertices whose minimum degree is larger than $(1/3-\epsilon)n$, and whose 
chromatic number is~$k$. On the other hand, they conjectured that every triangle-free 
graph on $n$ vertices whose minimum degree is larger than $n/3$ is $3$-colourable.
This was refuted by H\"aggkvist~\cite{H}, who found a $10$-regular triangle-free graph 
on $29$ vertices whose chromatic number is four. Jin~\cite{Jin} showed that this example 
is insofar optimal that every triangle-free graph whose minimum degree is strictly 
larger than~$10n/29$ has chromatic number three. Moreover, Thomassen~\cite{Th} proved 
that for every~$\epsilon>0$ there exists a constant~$c_\epsilon$ such that every 
triangle-free $n$-vertex graph with minimum degree at least $(1/3+\epsilon)n$ has 
chromatic number at most~$c_\epsilon$, and {\L}uczak~\cite{Lu06} supplemented this 
result by proving, roughly speaking, that for some constant $C_\eps$ there are at 
most $C_\eps$ `types' of such graphs. 
Finally, Brandt and Thomass\'e~\cite{BT} characterised all triangle-free graphs on~$n$ 
vertices whose minimum degree is larger than $n/3$; their theorem, stated in 
Section~\ref{sec:2141} below, plays a decisive r\^{o}le in the proof of our main result.

In this article, however, we mainly study triangle-free graphs $G$ with bounded 
independence number $\alpha(G)$. More specifically, we are interested in the behaviour 
of the function~$\ex(n, s)$, which for $n\ge s\ge 1$ gives the largest number of edges 
in a triangle-free graph on $n$ vertices whose independent sets have at most the 
size $s$, i.e., 
\begin{equation*}
	\ex(n,s)
	=
	\max\{e(G)\colon v(G)=n,\  G\not\supseteq K_3,\  \text{ and } \alpha(G)\le s \}\,.
\end{equation*}
Notice that Mantel's theorem yields
\[
	\ex(n,s)=\lfloor n^2/4\rfloor  
	\quad \text{ for every }
	s > \lfloor n/2\rfloor\,.
\]
Next we observe that in a triangle-free graph the neighbourhood of every vertex 
forms an independent set, which implies the so-called {\it trivial bound}
\[
	\ex(n,s)\le ns/2
\]
for all $n$ and $s$. Brandt~\cite{B10} provided several explicit constructions showing
that this upper bound is asymptotically optimal for $s\le n/3$, i.e., that we have  
\[
	\ex(n,s)=ns/2+o(n^2)
\]
in this range. Thus, it remains to study the behaviour of $\ex(n,s)$ for $s/n\in(1/3,1/2)$. 
 
This line of research was started over 50 years ago by Andr\'asfai~\cite{A}, 
who proved 
\[
	\ex(n,s)=n^2-4ns+5s^2 
	\quad \text{ for }
	s/n\in [2/5,1/2]\,.
\]
He also speculated that $\ex(n,s)$ might be a piecewise quadratic function with cusps at 
points of the form $s=kn/(3k-1)$. We slightly revised his conjecture in~\cite{LPR} and 
resolved the next case by showing 
\[
	\ex(n,s)=3n^2-15ns+20s^2 
	\quad \text{ for }
	s/n\in [3/8, 2/5]\,.
\]
The new version of the conjecture reads as follows. 
 
\begin{conj}\label{hip}
	If  $n/3<s\le n/2$, then 
	\begin{equation}\label{eq:2334}
		\ex(n,s)=\min_k g_k(n,s)\,,
	\end{equation}
	where
	\begin{equation}\label{eq2}
		g_k(n,s)=k(k-1)n^2/2-k(3k-4)ns+(3k-4)(3k-1)s^2/2
	\end{equation}
	for every $k\ge 1$. 
\end{conj}

Let us remark that we also have a conjecture on the extremal graphs for which 
equality holds in~\eqref{eq:2334}. Since the definition of these graphs requires some 
preparation, we state this stronger conjecture only at the end the article. 

As for the function $g(n, s)= \min_k g_k(n,s)$, which stands 
on the right side of~\eqref{eq:2334}, an elementary calculation (see~\cite{LPR}*{Cor.~2.7})
shows that for $k\ge 2$ and $\frac{k}{3k-1}n \le s < \frac{k-1}{3k-4}n$ we 
have $g(n, s)=g_k(n, s)$. Thus, for fixed $n$ the function $g(n, s)$ is piecewise quadratic 
in $s\in (n/3, n/2)$ with cusps at the points $s=kn/(3k-1)$ for $k\ge 2$. 

The main goal of this work is to add further plausibility to Conjecture~\ref{hip} 
by proving it whenever $s/n$ is slightly larger than one of the `critical points' $k/(3k-1)$.

\begin{theorem}\label{thm:main}
	For every $k\ge 2$ there exists $\gamma = \gamma(k)>0$ such that 
	\[
		\ex(n,s)=g_k(n,s)=\min_\ell g_\ell(n,s)
	\]
	whenever 
	\[	
		\frac{k}{3k-1}n \le  s \le \left(\frac{k}{3k-1}+\gamma\right)n\,.
	\]
	For instance, this holds for $\gamma(k)=(600k^6)^{-1}$.
\end{theorem}

Along the way, we establish the following minimum degree version of Conjecture~\ref{hip}.
 
\begin{theorem}\label{thm:mindeg}
	Let  $k\ge 2$ and $n\ge s \ge 1$. If $H$ denotes a triangle-free graph on~$n$
	vertices with $\alpha(H)\le s$ and 
	\[
		\delta(H) > \frac{k+1}{3k+2}n\,,
	\]
	then $e(H) \le g_k(n,s)$.  
\end{theorem}

Let us mention that similar problems could be and, in many cases, have been, 
considered for $K_r$-free graphs and, more generally, for $H$-free graphs for 
any given graph $H$. It hardly seems necessary to recall that Tur\'an's problem to 
determine the maximum number of edges in an $H$-free graph on $n$ vertices is 
fairly well understood thanks to the work of Tur\'an himself~\cite{T}, Erd\H{o}s,
Stone, and Simonovits~\cites{ES46, ESim}. The studies of the {\it chromatic threshold} 
(equal to 1/3 for triangle-free graphs by the aforementioned result of Thomassen~\cite{Th}) 
were begun by {\L}uczak and Thomass\'e~\cite{LT} and culminated in the work of 
Allen {\it et al.}~\cite{Aet} who determined this parameter for $H$-free graphs
when an arbitrary graph $H$ is given (for the precise definition of the `chromatic threshold'
we refer to either of those two articles).
 
The question on the behaviour of $\ex(n,s)$ considered in this work belongs to an area  
called {\it Ramsey-Tur\'an theory}, which has been initiated by Vera T.~S\'os and  
extensively investigated during the last fifty years. There is a comprehensive survey 
on this subject by S\'os and Simonovits~\cite{SS01}. Important milestones 
in the Ramsey-Tur\'an theory of general~\mbox{$K_r$-free} graphs were obtained by Bollob\'as, 
Erd\H{o}s, Hajnal, S\'os, Szemer\'edi~\cites{BE76, EHSS, ES69, Sz72}, and, more recently, 
by L\"{u}ders and Reiher~\cite{LR-a}. Due to their work, we asymptotically 
know the value of 
\[
	\ex_r(n,s)
	=
	\max\{e(G)\colon v(G)=n,\  G\not\supseteq K_r,\ \text{ and } \alpha(G)\le s \}
\]
for all $r\ge 3$ provided that $s/n \ll r^{-1}$ is sufficiently small. 
It would, of course, be interesting to study this function for larger values of $s/n$ as 
well, but, as the present article demonstrates, even the case $r=3$ of triangles seems 
to be fairly difficult.   

We believe that with considerably more work the methods of this article would allow 
to prove Theorem~\ref{thm:mindeg} under the weaker assumption $\delta(H)>n/3$ as well, which 
would lead to some numerical improvement on the value of $\gamma(k)$ in Theorem~\ref{thm:main}.

The structure of the article is the following. In Section~\ref{sec:1431} we start with 
the definition and some basic properties of the blow-up operation. The two subsequent  
sections define and study Andr\'asfai and Vega graphs, which -- by a result of 
Brandt and Thomass\'e -- 
are the 
main protagonists in the story of dense triangle-free graphs (see 
Theorem~\ref{th:dense} below). 
In particular, in this part of the article we prove some special cases of 
Theorem~\ref{thm:main} addressing blow-ups of these two types of graphs (see 
Lemma~\ref{ag0} and Lemma~\ref{l:vegaex}). In Section~\ref{sec:2141} these results 
will be employed in the proofs of the Theorems~\ref{thm:main} and~\ref{thm:mindeg}. 
Moreover, we shall state there a precise version of our conjecture on extremal cases 
in Conjecture~\ref{hip}.
These are the same as the extremal graphs for the two aforementioned lemmata, which 
we characterise in Lemma~\ref{lem:2130} and Lemma~\ref{l:vegaextremal}, respectively.

\section{Blow-ups of graphs}
\label{sec:1431}

Given a graph $F$ with vertex set $V(F)=\{v_1, \ldots, v_r\}$, a {\sl blow-up of $F$} 
is a graph $H$ obtained from $F$ upon replacing its vertices by 
independent sets $V_1, V_2, \ldots, V_r$, and each of its edges $v_iv_j\in E(F)$, 
$1\le i<j\le r$, by the complete bipartite graph $K(V_i,V_j)$ between $V_i$ and $V_j$. 
The sets $V_1, \ldots, V_r$ are called the {\sl vertex classes} of $H$. As above, we shall 
always denote the vertices of the original graph $F$ by lower case letters and the vertex 
classes of~$H$ by capitalised versions of the same letters. A blow-up is {\sl proper} if 
all vertex classes are non-empty and {\sl balanced} if all of them are of the same size. 
As the isomorphism type of~$H$ depends only on the sizes of its vertex classes it will be 
convenient to write $H=F(h)$, where the function $h\colon V(F)\lra \ZZ_{\ge 0}$ is defined 
by $h(v_i)=|V_i|$ for every $v_i\in V(F)$. In the special case where $h$ is the constant
function attaining always the value $p$ it will be convenient to write $H=F(p)$. 
For later use we remark that a blow-up of a blow-up of $F$ is again a blow-up of $F$.

Resuming the discussion of the blow-up $H$ of $F$ with vertex classes $V_1, \ldots, V_r$ 
we set 
\[
	N(V_i)=\bigcup_{v_j\in \Nn(v_i)} V_j
\]
for each of these vertex classes, where $\Nn(v_i)$ denotes the neighbourhood of $v_i$ 
in $F$. Clearly, all vertices in $V_i$ have the neighbourhood $N(V_i)$ in $H$ and,
consequently, every non-empty vertex class $V_i$ satisfies
\begin{equation}\label{eq:neighbours}
	\delta(H) \le |N(V_i)| \le \Delta (H)\,.
\end{equation}

Now simple averaging leads to the following observation. 

\begin{fact}\label{fact:1641}
	Let $H$ be an $n$-vertex blow-up of a $k$-regular graph $F$ on the $r$-element 
	vertex set $V(F)=\{v_1, \ldots, v_r\}$. 
	\begin{enumerate}[label=\alabel]
		\item\label{it:1641a} We have 
				\begin{equation}\label{eq:1734}
					\sum_{i=1}^r |N(V_i)|=kn\,.
				\end{equation}
		\item\label{it:1641b} If the blow-up $H$ is proper, then 
		\[
			\delta(H)\le \frac{kn}r \le \Delta(H)\,.
		\]
	\end{enumerate}
\end{fact}

\begin{proof}
	The double counting argument 
	\[
		\sum_{i=1}^r |N(V_i)|
		=
		\sum_{i=1}^r \sum_{v_j\in \Nn(v_i)} |V_j|
		=
		\sum_{j=1}^r |\Nn(v_j)|\,|V_j|
		=
		k\sum _{j=1}^r |V_j|
		=
		kn
	\]
	establishes part~\ref{it:1641a}. If $H$ is a proper blow-up, then~\eqref{eq:neighbours} 
	yields 
	\[
		r\delta(H) \le \sum_{i=1}^r |N(V_i)| \le r\Delta(H)
	\]
	and part~\ref{it:1641b} follows. 
\end{proof}	

The second part of the foregoing fact has the following useful consequence.
  
\begin{lemma}\label{l:kregular}
	Let $d, k\ge 2$ be two integers and let $J$ be a graph. Suppose that $J$ has a $k$-regular
	proper blow-up $F$ on $3k-1$ vertices and that $H$ denotes a further proper blow-up of $J$
	having $n$ vertices. If 
	\begin{equation}\label{eq:1710}
		\frac{d+1}{3d+2}n < \delta(H)
		\quad \text{ and } \quad
		\Delta(H) < \frac{d-1}{3d-4}n\,,
	\end{equation}
	then $k=d$.
\end{lemma}

\begin{proof}
	Notice that every balanced blow-up $H(p)$ of $H$ is a proper blow-up of $J$ as well. 
	By applying this observation to a sufficiently large integer $p$ we learn that we may 
	assume, without loss of generality, that $H$ is a proper blow-up of $F$. 
	Now Fact~\ref{fact:1641}\ref{it:1641b} reveals  
	\[
		\delta(H)\le \frac{kn}{3k-1} \le \Delta(H)\,,
	\]
	which together with~\eqref{eq:1710} implies
	\[
		\frac{d+1}{3d+2} < \frac{k}{3k-1} < \frac{d-1}{3d-4}\,.
	\]
	Taking reciprocals we obtain 
	\[
		3-\frac{1}{d-1} < 3-\frac 1k < 3-\frac 1{d+1}\,,
	\]
	i.e., $d-1<k<d+1$. 
\end{proof}

The main result of this section is the following upper bound on the number of edges 
of a blow-up. 

\begin{lemma}\label{edges}
	Suppose that $F$ is a $k$-regular graph on $r$ vertices
	and  that $H$ is an $n$-vertex blow-up of $F$ all of whose vertex classes have at 
	least the size $x$. If we have $|N(Z)|\le s$ for every vertex class $Z$ of $H$, then
	\begin{equation}\label{eq:lowx}
		e(H) \le \frac{ns}2 -\frac{x(rs- kn)}2\,.
	\end{equation}
	Moreover, if $rs\neq kn$ and~\eqref{eq:lowx} holds with equality, then  
	\begin{enumerate}
		\item[$\bullet$] at least one vertex class of $H$ has size $x$; 
		\item[$\bullet$] every vertex class $V_i$ of $H$ with $|V_i| > x$
			satisfies $|N(V_i)|=s$.
	\end{enumerate}
\end{lemma}
\begin{proof} 
	As usual, we write the vertex set of $F$ in the form $V(F)=\{v_1, \ldots, v_r\}$. 
	Setting $x_i=|V_i|$ for $i\in [r]$ we have 
	\begin{align*}
		2e(H) 
		&= 
		\sum_{v\in V(H)}\deg_H(v)
		=
		\sum_{i=1}^r|V_i|\cdot|N(V_i)| \\
		&= 
		\sum_{i=1}^r x\cdot|N(V_i)|+\sum_{i=1}^r(x_i-x)\cdot|N(V_i)| \\
		&\le 
		x\sum_{i=1}^r|N(V_i)|+\sum_{i=1}^r(x_i-x)s
		\overset{\eqref{eq:1734}}{=}
		xkn+s(n-xr)
		=
		ns-x(rs-kn)\,, 
	\end{align*}	
	which proves the desired upper bound on $e(H)$.
	
	Suppose from now on that this estimate holds with equality and that $rs\ne kn$. 
	This means that $(x_i-x)|N(V_i)|=(x_i-x)s$ holds for every $i\in [r]$,
	or in other words that $x_i>x$ implies $|N(V_i)|=s$. This proves the second bullet.
	Now, if the first bullet fails, we have $|N(V_i)|=s$ for every $i\in [r]$ 
	and~\eqref{eq:1734} yields the contradiction $rs=kn$.
\end{proof}

Let us notice the following consequence of the above result.

\begin{cor}\label{blowupedges}
	Suppose that $k\ge 2$ is a natural number, $F$ is a $k$-regular graph 
	on~$3k-1$ vertices and $H$
	is an $n$-vertex blow-up of $F$. If $|N(Z)|\le s$ and $|Z|\ge (k-1)n - (3k-4)s$
	hold for every vertex class $Z$ of $H$, then 
	\[
		e(H)\le g_k(n,s)\,.
	\]
	Moreover, if $e(H)=g_k(n,s)$, then 
	\begin{enumerate}[label=\rmlabel]
	\item\label{it:2140a} $kn/(3k-1) \le s \le (k-1)n/(3k-4)$;
	\item\label{it:2140b} $H$ contains a vertex class of size $(k-1)n - (3k-4)s$;
	\item\label{it:2140c} the neighbourhood of each vertex class of $H$ containing 
		more than $(k-1)n - (3k-4)s$ vertices has size $s$.
	\end{enumerate}			
\end{cor}
 
\begin{proof}
	In order to prove the desired upper bound on $e(H)$ we remark that Lemma~\ref{edges} 
	applied to $r=3k-1$ and $x=(k-1)n-(3k-4)s$ yields  
	\[
		e(H)
		\le 
		\tfrac12\bigl[ns-\bigl((k-1)n -(3k-4)s\bigr)\bigl((3k-1)s - kn\bigr)\bigr]
		=
		g_k(n,s)\,. 
	\]
	Let us now study the case that equality holds in this estimate. If clause~\ref{it:2140a} 
	failed, then the trivial upper bound $e(H)\le ns/2$ would contradict $e(H)=g_k(n, s)$.
	
	The remaining two clauses follow from the moreover-part in Lemma~\ref{edges} provided
	that its assumption $(3k-1)s\ne kn$ holds. So it remains to deal with the 
	case $s=kn/(3k-1)$. Now $(k-1)n - (3k-4)s = n/(3k-1)$ is at the same time the average 
	size of the vertex classes of $H$ and a lower bound on the sizes of these vertex classes.
	In other words, $H$ is the balanced blow-up $F\bigl(n/(3k-1)\bigr)$ 
	and~\ref{it:2140b},~\ref{it:2140c} hold trivially. 
\end{proof}

\section{Andr\'asfai graphs and their blow-ups}

The characterisation of triangle-free graphs on $n$ vertices whose minimum degree is larger 
than $n/3$ due to Brandt and Thomass\'e~\cite{BT} involves two explicit families 
of such graphs, called {\it Andr\'asfai graphs} and {\it Vega graphs} (see 
Theorem~\ref{th:dense} below). In this section we study the first of these graph sequences,
which has been introduced by Andr\'asfai in~\cite{A} and has been rediscovered several 
times throughout the years. 

One way to construct a $k$-regular triangle-free graph is to take an Abelian 
group $\GG$, a symmetric sum-free subset $S$ of size $k$, and to form the 
Cayley graph $\textrm{Cayley}(\GG; S)$. A natural (and, as we shall soon argue, 
generic) example occurs when we take the cyclic group $\ZZ_{3k-1}$ and its 
sum-free subset $S_k=\{k, k+1, \ldots, 2k-1\}$. 
The \emph{Andr\'asfai graph~$\G_k$} is defined to be the corresponding Cayley graph 
$\textrm{Cayley}(\ZZ_{3k-1}; S_k)$. Describing the same graph in more concrete terms,
we set 
\[
	V(\G_k)=\{v_0, v_1, \dots, v_{3k-2}\}
\]
and declare the adjacencies in $\Gamma_k$ by
\begin{equation}\label{eq:1200}
	v_iv_j\in E(\G_k)
	\,\,\, \Longleftrightarrow \,\,\,	
	k \le |i-j|\le 2k-1
\end{equation}
for all vertices $v_i, v_j\in V(\G_k)$.
For instance, $\G_1=K_2$, $\G_2 = C_5$, and Figure~\ref{fig:ag} shows some further 
Andr\'asfai graphs.
\begin{figure}[ht]
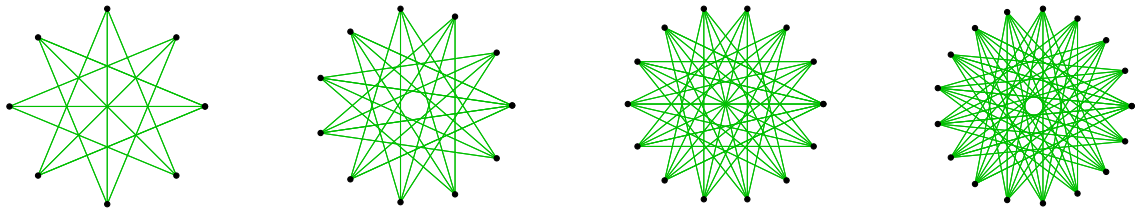

	\centering
	\begin{multicols}{4}
		\ag{3} \\ \ag{4} \\ \ag{5} \\ \ag{6}
	\end{multicols}
	\caption{Andr\'asfai graphs  
		$\G_3$, $\G_4$,  $\G_5$, and $\G_6$.
	}
	\label{fig:ag}
\end{figure}

Let us remark that given an Abelian group $\GG$ containing a sum-free set $S$ 
with $|S|>|\GG|/3$ one can show by means of Kneser's theorem (see~\cite{LS}) that 
for some positive integer~$k$ there 
exists a homomorphism $\varphi\colon\GG\longrightarrow \ZZ_{3k-1}$
satisfying $\varphi[S]\subseteq S_k$. 
Such a group homomorphism~$\varphi$ induces a graph homomorphism 
$\varphi_\star\colon \textrm{Cayley}(\GG;S)\longrightarrow \Gamma_k$,
or, in other words, it indicates that $\textrm{Cayley}(\GG;S)$ is contained in a sufficiently 
large blow-up of $\Gamma_k$. These considerations reveal that balanced blow-ups of 
Andr\'asfai graphs are universal in the class of triangle-free Cayley graphs whose 
density is larger than $1/3$. Somewhat relatedly, a finite graph is a subgraph of 
a blow-up of an Andr\'asfai graph if and only if it is isomorphic to a subgraph of the 
infinite triangle-free Cayley graph $\textrm{Cayley}\bigl(\RR/\ZZ;(1/3, 2/3)+\ZZ\bigr)$,
see~\cite{BMPP2018}*{Lemma~2.1}.
We proceed with three well known, useful properties of Andr\'asfai graphs.

\begin{fact}\label{fact:A1} 
	Let $k\ge 2$ be an integer.
	\begin{enumerate}[label=\rmlabel]
		\item\label{it:1112a} We have $\G_{k-1}\subseteq \G_{k}$. 
			Conversely, if we remove a vertex from $\G_{k}$, then the resulting 
			graph is a subgraph of a proper blow-up of $\G_{k-1}$. Consequently, 
			every subgraph of~$\Gamma_k$ is a subgraph of a proper blow-up of $\Gamma_\ell$ 
			for some $\ell\le k$.
		\item\label{it:1112b} The chromatic number of $\G_k$ is $3$. 
		\item\label{it:1112c} Every independent set in $\G_k$ is contained 
			in the neighbourhood of some vertex of $\G_k$. In particular, $\alpha(\G_k)=k$.
	\end{enumerate}
\end{fact} 

\begin{proof}
	The first part of~\ref{it:1112a} follows from $\G_{k-1}\cong \G_k - \{v_0, v_k, v_{2k}\}$.
	Furthermore, the graph $\G_k^-=\G_k-\{v_{k}\}$ 
	satisfies $\Nn(v_0, \G_k^-)\subseteq \Nn(v_1, \G_k^-)$
	and $\Nn(v_{2k}, \G_k^-)\subseteq \Nn(v_{2k-1}, \G_k^-)$,
	which proves~$\G_k^-$ to be a subgraph of the blow-up 
	of $\G_k - \{v_0, v_k, v_{2k}\}$ obtained by doubling the vertices~$v_1$
	and~$v_{2k-1}$. This establishes the second part of~\ref{it:1112a} and the 
	last part follows inductively.  
	
	Proceeding with~\ref{it:1112b} we observe that, due to~\ref{it:1112a}, 
	$\G_k$ contains a pentagon~$\G_2$ as a subgraph, whence $\chi(\G_k)\ge \chi(\G_2)=3$. 
	To verify the reverse inequality, we partition $V(\G_k)$ into the three independent sets 
	$\{v_0, v_1, \ldots, v_{k-1}\}$, $\{v_k, v_{k+1},\ldots, v_{2k-1}\}$, and 
	$\{v_{2k}, v_{2k+1}, \ldots, v_{3k-2}\}$.

	Finally, let $S\subseteq V(\G_k)$ be a nonempty independent set we want to cover by 
	the neighbourhood of an appropriate vertex. By symmetry we may suppose that $v_k\in S$, 
	which due to~\eqref{eq:1200} entails $S\subseteq \{v_1, v_2, \ldots, v_{2k-1}\}$. 
	Let $i, j\in [2k-1]$ be the smallest and largest index with $v_i, v_j\in S$. 
	Now $S\subseteq \{v_i, \ldots, v_j\}$, a further application of~\eqref{eq:1200} 
	shows $j-i \le k-1$, and altogether we have $S\subseteq \Nn(v_{j+k})$.
\end{proof}	

Let us consider for $k\ge 2$ and $n\ge s\ge 1$ an $n$-vertex blow-up $H$ of the 
Andr\'asfai graph~$\G_k$ with $\al(H)\le s$. According to the notation of 
Section~\ref{sec:1431}, this blow-up comes together with a fixed partition 
\[
	V(H)=V_0\dcup V_1\dcup \ldots \dcup V_{3k-2}
\]
of its vertex set. Clearly, since $\G_k$ is triangle-free, so is $H$. Therefore the neighbourhood of each vertex class of~$H$ is an independent set, which proves
\begin{equation}\label{eq:inde}
	|N(V_i)|=|V_{i+k}|+\ldots+|V_{i+2k-1}| \le s 
\end{equation}
for every $i\in\ZZ_{3k-1}$.	We use this inequality to bound the sizes of the vertex classes of $H$. 

\begin{fact}\label{Vi}
	For every $i\in\ZZ_{3k-1}$ we have
	\[
		(k-1)n - (3k-4)s \le |V_i|\le 3s-n\,.
	\]
\end{fact}
\begin{proof}
	The upper bound holds because of
	\begin{align*}
		|V_i| &= |V_0 \cup V_1 \cup \dots \cup V_{3k-2}|+|V_i|-n \\
		&= |N(V_i)| + |N(V_{i+k})|+|N(V_{i-k})|-n 
		\overset{\eqref{eq:inde}}{\le} 3s-n\,.
	\end{align*}
	By applying this estimate to $V_{i+k+1}, \ldots, V_{i+2k-2}$ instead of $V_i$ 
	we infer 
	\begin{align*}
		|V_i| &= n-\bigl(|N(V_{i+(k-1)})|+|N(V_{i-(k-1)})|+
			|V_{i+k+1}|+\dots+|V_{i+2k-2}|\bigr) \\
		&\overset{\eqref{eq:inde}}{\ge} n-\bigl(2s +(k-2)(3s-n)\bigr)= (k-1)n-(3k-4)s\,. \qedhere
	\end{align*}
\end{proof}

By combining Corollary~\ref{blowupedges} and Fact~\ref{Vi} we arrive at the main 
result of this section.

\begin{lemma}\label{ag0}
	Let $k$ be a natural number.
	If for $n\ge s\ge 1$ the graph $H$ is an $n$-vertex blow-up of the Andr\'asfai graph 
	$\G_k$ satisfying $\al(H)\le s$, then 
	\[
		\pushQED{\qed} 
		e(H)\le  g_k(n,s)\,.\qedhere
		\popQED
	\]
\end{lemma}

As remarked in~\cite{LPR}*{Fact~1.5}, this estimate can hold with equality.
We would now like to complement this observation by an explicit description
of all extremal cases. As it turns out, every $n$-vertex blow-up $H$ of $\G_k$ 
with $\al(H)\le s$ and $e(H)=g_k(n,s)$ belongs to the following family of graphs. 

\begin{dfn}\label{def:exan}
	Given natural numbers $k\ge 2$, $n$, and $s\in [kn/(3k-1), (k-1)n/(3k-4)]$
	the family $\cG^n_k(s)$ consists of all graphs obtained from $\G_k$ by blowing up 
	its vertices
	\begin{enumerate} 
		\item[$\bullet$] $v_0$ and $v_k$ by $(k-1)n-(3k-4)s$, 
		\item[$\bullet$] $v_{2k-1}$ by $a$, $v_{2k}$ by $b$, where 
				$a,b\in [(k-1)n-(3k-4)s, 3s-n]$ are two integers summing up 
				to $(k-2)n-(3k-7)s$,
		\item[$\bullet$] and all remaining vertices by $3s-n$. 
	\end{enumerate}
\end{dfn}

Observe that if $s=kn/(3k-1)$, then $(k-1)n-(3k-4)s = 3s-n$ and the only graph 
in~$\cG^n_k(s)$ is the balanced blow-up $\G_k(3s-n)$.
At the other end of the spectrum we have the case $s=(k-1)n/(3k-4)$, where 
$(k-1)n-(3k-4)s=0$ and the graphs in~$\cG^n_k(s)$ are actually blow-ups 
of $\G_k\sm\{v_0, v_k\}$. In this graph, $v_{2k}$ is a twin sister of $v_{2k-1}$
and again the class $\cG^n_k(s)$ consists of a single graph. In view of 
$\G_{k-1}\cong \G_k - \{v_0, v_k, v_{2k}\}$ this graph is the balanced blow-up 
$\G_{k-1}(3s-n)$. However, if~$s$ lies strictly between $kn/(3k-1)$
and $(k-1)n/(3k-4)$, then $\cG^n_k(s)$ consists 
of $\big\lfloor\bigl((3k-1)s-kn\bigr)/2\big\rfloor+1$ mutually non-isomorphic graphs.  
		
\begin{lemma}\label{lem:2130}
	If $H$ denotes an $n$-vertex blow-up of $\G_k$ satisfying 
	\[
		\alpha(H)\le s
		\quad \text{ and } \quad
		e(H)= g_k(n,s)
	\]
	for some $k\ge 2$, then $kn/(3k-1)\le s\le (k-1)n/(3k-4)$
	and $H$ is isomorphic to some graph in $\cG^n_k(s)$.
\end{lemma}
\begin{proof}
	As usual we denote the vertex classes of $H$ corresponding to the vertices of $\G_k$
	by $V_0, \ldots, V_{3k-2}$. Recall that by Fact~\ref{Vi} 
	we have 
	\begin{equation}\label{eq:1847}
		(k-1)n - (3k-4)s \le |V_i|\le 3s-n
		\quad \text{ for every } i\in\ZZ_{3k-1}\,.
	\end{equation}
	In particular, $H$ has the properties enumerated in the moreover-part of
	Corollary~\ref{blowupedges} and clause~\ref{it:2140a} corresponds to the estimates on $s$ 
	stated in the lemma. Next,~\ref{it:2140b} allows us to assume, without loss of generality, that 
	\begin{equation}\label{eq:v1}
			|V_0| = (k-1)n - (3k-4)s\,.
	\end{equation}
	In the special case 
	\[
		|V_k|=|V_{2k-1}|=(k-1)n - (3k-4)s
	\]
	we have 
	\[
		n=|V_0|+\ldots +|V_{3k-2}|
		\overset{\eqref{eq:1847}}{\le}
		3\bigl((k-1)n - (3k-4)s\bigr)+(3k-4)(3s-n)=n\,,
	\]
	which yields $|V_i|=3s-n$ for every $i\ne 0, k, 2k-1$, meaning that 
	$a=(k-1)n - (3k-4)s$ exemplifies $H\in \cG^n_k(s)$. By symmetry we may 
	therefore suppose $|V_{2k-1}|>(k-1)n - (3k-4)s$ from now on, which  
	in view of Corollary~\ref{blowupedges}\ref{it:2140c} entails
	\[
		s=|N(V_{2k-1})|=|V_0|+\ldots+|V_{k-1}|
		\overset{\eqref{eq:1847}}{\le}
		\bigl((k-1)n - (3k-4)s\bigr)+(k-1)(3s-n)=s\,,
	\]
	i.e., 
	\begin{equation}\label{eq:1903}
		|V_1|=\ldots=|V_{k-1}|=3s-n\,.
	\end{equation}
	Because of $|V_1|+\ldots+|V_k|=|N(V_{2k})|\le s$ and~\eqref{eq:1847} this yields 
	\begin{equation}\label{eq:1911}
		|V_k|=(k-1)n - (3k-4)s\,,
	\end{equation}
	which together with~\eqref{eq:v1} establishes the first bullet in 
	Definition~\ref{def:exan}.
	
	Next we observe that, since by~\eqref{eq:v1} the lower bound on $|V_0|$ provided by 
	Fact~\ref{Vi} holds with equality, an easy inspection of the proof of Fact~\ref{Vi}
	discloses 
		\[
			|V_{k+1}|=\ldots=|V_{2k-2}|=3s-n
		\]
	and for the same reason~\eqref{eq:1911} leads to $|V_{2k+1}|=\ldots=|V_{3k-2}|=3s-n$.
	Combined with~\eqref{eq:1903} these equations confirm the third bullet in
	Definition~\ref{def:exan} and, finally, the second bullet follows easily in the 
	light of $v(H)=n$ and~\eqref{eq:1847}.
\end{proof}

\section{Vega graphs and their blow-ups}

In this section we investigate another important class of dense triangle-free graphs which, 
unlike Andr\'asfai graphs, have  chromatic number four.
Let us recall that the first $4$-chromatic triangle-free graph on $n$ vertices whose minimum 
degree is larger than $n/3$ was a blow-up of the Gr\"otzsch graph discovered in 1981
by H\"aggkvist~\cite{H}. In 1998, Brandt and Pisanski~\cite{BP} worked with a computer 
program named Vega and found an infinite sequence of $4$-chromatic triangle-free graphs 
admitting such blow-ups (see Fact \ref{f:G1}). Due to their origin, these graphs are 
called \emph{Vega graphs}. 

\begin{figure}[ht]
	\centering
	\begin{tikzpicture}[scale=.8]
	
	\coordinate (v) at (30:4cm);
	\coordinate (a) at (90:4cm);
	\coordinate (w) at (150:4cm);
	\coordinate (b) at (210:4cm);
	\coordinate (u) at (270:4cm);
	\coordinate (c) at (330:4cm);
	
	\coordinate (x) at (7,3);
	\coordinate (y) at (7,-3);

	\foreach \i in {1,...,8}{
		\coordinate (v\i) at (135-\i*45:2cm);
	}
	
	\draw (x)--(y);
	\draw [color=violet!75!black, thick](c)--(x)--(a);
	\draw [color=orange!75!black, thick](v)--(y)--(u);
	\draw [rounded corners=35, color=violet!75!black, thick] (x) -- (0,5.5)--(-5,2)--(b);
	\draw [rounded corners=35, color=orange!75!black, thick] (y) -- (0,-5.5)--(-5,-2)--(w);

	\foreach \i in {4,...,6}{	
		\draw (v\i)--(v1);
	}		
	\foreach \i in {5,...,7}{	
		\draw (v\i)--(v2);
	}		
	\foreach \i in {6,...,8}{	
		\draw (v\i)--(v3);
	}		
	\foreach \i in {3,...,5}{	
		\draw  (v\i)--(v8);
	}		
	\draw (v4)--(v7);
	\draw  (a)--(v)--(c)--(u)--(b)--(w)--(a);

	\begin{pgfonlayer}{front}
	
	\foreach \i in {a,b,c}{
		\draw[violet!75!black, very thick]  (\i) circle (4pt);
	}	
	
	\foreach \i in {v, u, w}{
		\draw[orange!75!black, very thick]  (\i) circle (4pt);
	}	
	
	\foreach \i in {c, w}{
		\fill[blue!75!white]  (\i) circle (4pt);
	}	
	\foreach \i in {a,u}{
		\fill[red!75!white]  (\i) circle (4pt);
	}	
	\foreach \i in {v,b}{
		\fill[green!75!white]  (\i) circle (4pt);
	}	
	
	\foreach \i in {1,...,3}{
		\draw[red!75!black, very thick]  (v\i) circle (4pt);
		\fill[red!75!white]  (v\i) circle (4pt);}
	
	\foreach \i in {4,...,6}{
		\draw[green!75!black, very thick]  (v\i) circle (4pt);
		\fill[green!75!white]  (v\i) circle (4pt);}
	
	\foreach \i in {7,8}{
		\draw[blue!75!black, very thick]  (v\i) circle (4pt);
		\fill[blue!75!white]  (v\i) circle (4pt);}

	\draw[violet!75!black, very thick]  (x) circle (4pt);
	\fill[violet!75!white]  (x) circle (4pt);
	
	\draw[orange!75!black, very thick]  (y) circle (4pt);
	\fill[orange!75!white]  (y) circle (4pt);
	
	\node at ($(a)+(0,-.4)$) {$a$};
	\node at ($(b)+(.4,.3)$) {$b$};
	\node at ($(c)+(-.4,.3)$) {$c$};
	\node at ($(v)+(-.4,-.3)$) {$v$};
	\node at ($(w)+(.4,-.3)$) {$w$};
	\node at ($(u)+(0,.4)$) {$u$};
	\node at ($(x)+(.4,.1)$) {$x$};
	\node at ($(y)+(.4,-.1)$) {$y$};
	
	\node at ($(v1)+(.4,.2)$) {$v_0$};
	\node at ($(v3)+(.3,-.4)$) {$v_{i-1}$};
	\node at ($(v4)+(0,-.4)$) {$v_{i}$};
	\node at ($(v6)+(0,-.4)$) {$v_{2i-1}$};
	\node at ($(v7)+(-.2,-.3)$) {$v_{2i}$};
	\node at ($(v8)+(0,.4)$) {$v_{3i-2}$};

	\end{pgfonlayer}
	
	\end{tikzpicture}
	\caption{The Vega graph $\Upsilon^{00}_i$. The vertices of the external 6-cycle $\mathcal C_6$ 
		are connected with the vertices of the same colour of the Andr\'asfai graph $\Gamma_i$ in the middle. }
	\label{fig:Vega}
\end{figure}
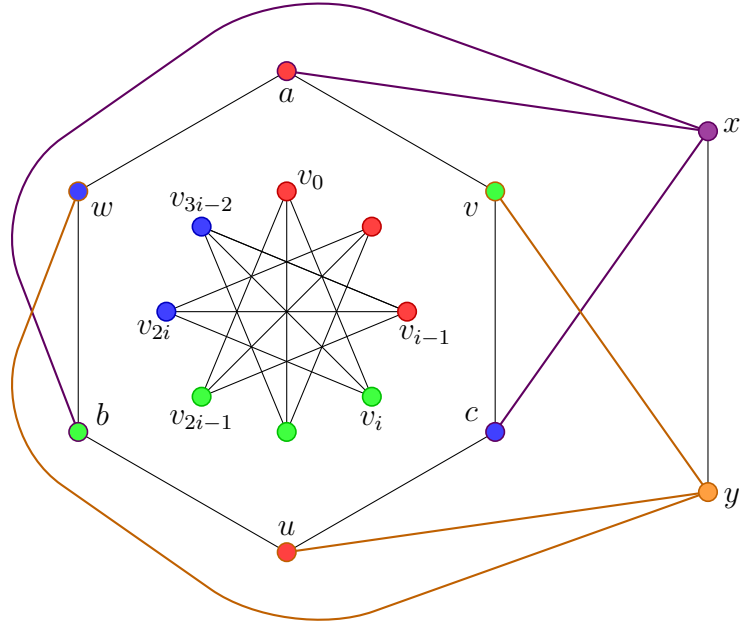

\subsection{Definitions and main results.}
We commence by presenting a construction of Vega graphs following the work of Brandt and 
Thomass\'e~\cite{BT}. Let an integer $i\ge 2$ be given. Start with an Andr\'asfai graph~$\G_i$ 
on the vertex set $\{v_0, \ldots,v_{3i-2}\}$ and add an edge $xy$ together with an 
induced $6$-cycle $avcubw$ such that $x$ is joined to $a,b,c$ and $y$ is joined to $u,v,w$. 
Moreover, connect 
\begin{enumerate}
	\item[$\bullet$] $a$, $u$ to $\{v_0, \ldots, v_{i-1}\}$,
	\item[$\bullet$] $b$, $v$ to $\{v_i, \ldots, v_{2i-1}\}$, 
	\item[$\bullet$] and $c$, $w$ to $\{v_{2i}, \ldots, v_{3i-2}\}$.
\end{enumerate}
This completes the definition of the sole Vega graph on $3i + 7$ vertices, 
which we denote by~$\Upsilon_i^{00}$ and sometimes just by $\Upsilon_i$ 
(see Figure \ref{fig:Vega}).

There are two Vega graphs on $3i + 6$ vertices obtainable from $\Upsilon_i^{00}$ by 
simple vertex deletions, namely $\Upsilon_i^{10}=\Upsilon_i^{00} - \{y\}$ and 
$\Upsilon_i^{01}=\Upsilon_i^{00} - \{v_{2i-1}\}$. Finally, the last Vega graph, 
$\Upsilon_i^{11}=\Upsilon_i^{00} - \{y, v_{2i-1}\}$, has $3i + 5$ vertices. 
Observe that the vertex $y$ is present in $\Upsilon_i^{\mu\nu}$ if and only if $\mu=0$. 
Similarly, $\nu=1$ in $\Upsilon_i^{\mu\nu}$ indicates the absence of the vertex $v_{2i-1}$. 
For instance,~$\Upsilon_2^{11}$ is the well known Gr\"otzsch graph. For later use 
we would like to remark that~$\Upsilon_2^{11}\subseteq \Upsilon_i^{\mu\nu}$ gives a  
quick proof of the aforementioned estimate $\chi(\Upsilon_i^{\mu\nu})\ge 4$. 
The main result of this section reads as follows.

\begin{lemma}\label{l:vegaex}
	Let integers $i\ge 2$ and $\mu, \nu \in \{0,1\}$ be given and set $k=9i-(6+\mu+\nu)$.
	If for $n\ge s\ge 1$ the graph $H$ is an $n$-vertex blow-up of the Vega graph $\ve$
	satisfying $\al(H)\le s$, then 
	\[
		e(H)\le g_k(n,s)\,.
	\]
\end{lemma}

The proof of the above lemma is based on Corollary \ref{blowupedges}, which will become 
applicable once we have exhibited a $k$-regular blow-up of $\ve$ on $3k-1$ 
vertices. 
To this end we shall use an appropriate weight function 
\[
	\omega_{\mu\nu}\colon V(\Upsilon_i^{}) \longrightarrow \ZZ_{\ge 0}
\]
from~\cite{BT}. In the special case $\mu=\nu=0$ this function is defined by the following 
table.

\smallskip

\begin{center}
\begin{tabular}{|c|c|c|c|c|c|}
\hline 
vertex $z$     &$x$, $y$&$a$, $b$, $u$, $v$&$c$, $w$&$v_0$, $v_{2i-1}$& $v_j$ (where $j\ne 0, 2i-1$) 
\\ \hline
weight $\omega_{00}(z)$ &$1$     &$3i-2$            &$3i-3$  &$1$           &$3$ \\ \hline 
\end{tabular}
\end{center}

\smallskip

In general, one uses the function 
\begin{equation}\label{eq:1016}
	\omega_{\mu\nu}=\omega_{00}-\mu f -\nu g\,,
\end{equation}
where $f, g\colon V(\Upsilon_i^{}) \longrightarrow \ZZ$ are defined by 
\[
	f(z)=\begin{cases}
		1    & \text {if } z=u, v, w, y \cr
		-1   & \text {if } z=x \cr
		0    & \text {otherwise } 
	\end{cases}
\]
and
\[
	g(z)=\begin{cases}
		1    & \text {if } z=b, v, v_{i-1}, v_{2i-1} \cr
		-1   & \text {if } z=v_0 \cr
		0    & \text {otherwise. } 
	\end{cases}
\]

The weight function~$\omega_{\mu\nu}$ is visualised in Figure~\ref{fig:Vega_w}.

	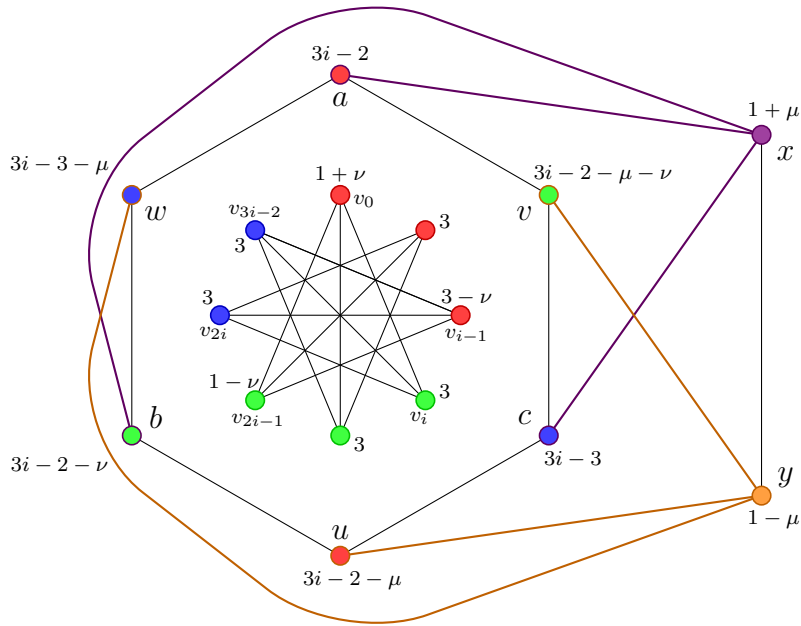
\begin{figure}[ht]
		\centering
		\begin{tikzpicture}[scale=.8]
		
		\coordinate (v) at (30:4cm);
		\coordinate (a) at (90:4cm);
		\coordinate (w) at (150:4cm);
		\coordinate (b) at (210:4cm);
		\coordinate (u) at (270:4cm);
		\coordinate (c) at (330:4cm);
		
		\coordinate (x) at (7,3);
		\coordinate (y) at (7,-3);

		\foreach \i in {1,...,8}{
			\coordinate (v\i) at (135-\i*45:2cm);
		}
		
		\draw (x)--(y);
		\draw [color=violet!75!black, thick](c)--(x)--(a);
		\draw [color=orange!75!black, thick](v)--(y)--(u);
		\draw [rounded corners=35, color=violet!75!black, thick] (x) -- (0,5.5)--(-4.5,2)--(b);
		\draw [rounded corners=35, color=orange!75!black, thick] (y) -- (0,-5.5)--(-4.5,-2)--(w);

		\foreach \i in {4,...,6}{	
			\draw (v\i)--(v1);
		}		
		\foreach \i in {5,...,7}{	
			\draw (v\i)--(v2);
		}		
		\foreach \i in {6,...,8}{	
			\draw (v\i)--(v3);
		}		
		\foreach \i in {3,...,5}{	
			\draw  (v\i)--(v8);
		}		
		\draw (v4)--(v7);
		\draw  (a)--(v)--(c)--(u)--(b)--(w)--(a);

		\begin{pgfonlayer}{front}
		
		\foreach \i in {a,b,c}{
			\draw[violet!75!black, very thick]  (\i) circle (4pt);
		}	
		
		\foreach \i in {v, u, w}{
			\draw[orange!75!black, very thick]  (\i) circle (4pt);
		}	
		
		\foreach \i in {c, w}{
			\fill[blue!75!white]  (\i) circle (4pt);
		}	
		\foreach \i in {a,u}{
			\fill[red!75!white]  (\i) circle (4pt);
		}	
		\foreach \i in {v,b}{
			\fill[green!75!white]  (\i) circle (4pt);
		}	
		
		\foreach \i in {1,...,3}{
			\draw[red!75!black, very thick]  (v\i) circle (4pt);
			\fill[red!75!white]  (v\i) circle (4pt);}
		
		\foreach \i in {4,...,6}{
			\draw[green!75!black, very thick]  (v\i) circle (4pt);
			\fill[green!75!white]  (v\i) circle (4pt);}
		
		\foreach \i in {7,8}{
			\draw[blue!75!black, very thick]  (v\i) circle (4pt);
			\fill[blue!75!white]  (v\i) circle (4pt);}

		\draw[violet!75!black, very thick]  (x) circle (4pt);
		\fill[violet!75!white]  (x) circle (4pt);
		
		\draw[orange!75!black, very thick]  (y) circle (4pt);
		\fill[orange!75!white]  (y) circle (4pt);
		
		\node at ($(a)+(0,.35)$) {{\tiny $3i-2$}};
		\node at ($(b)+(-1.2,-.5)$) {{\tiny $3i-2-\nu$}};
		\node at ($(c)+(.4,-.4)$) {{\tiny $3i-3$}};
		\node at ($(v)+(.9,.35)$) {{\tiny $3i-2-\mu-\nu$}};
		\node at ($(w)+(-1.2,.5)$) {{\tiny $3i-3-\mu$}};
		\node at ($(u)+(.2,-.4)$) {{\tiny $3i-2-\mu$}};
		\node at ($(x)+(.2,.4)$) {{\tiny $1+\mu$}};
		\node at ($(y)+(.2,-.4)$) {{\tiny $1-\mu$}};
		
		\node at ($(v1)+(.0,.35)$) {{\tiny $1+\nu$}};
		\node at ($(v3)+(.1,.3)$) {{\tiny $3-\nu$}};
		\node at ($(v2)+(.32,.15)$) {{\tiny $3$}};
		\node at ($(v4)+(.32,.15)$) {{\tiny $3$}};
		\node at ($(v6)+(-.35,.3)$) {{\tiny $1-\nu$}};
		\node at ($(v7)+(-.2,.3)$) {{\tiny $3$}};
		\node at ($(v8)+(-.25,-.25)$) {{\tiny $3$}};
		\node at ($(v5)+(.32,-.15)$) {{\tiny $3$}};

		\node at ($(a)+(0,-.4)$) {$a$};
		\node at ($(b)+(.4,.3)$) {$b$};
		\node at ($(c)+(-.4,.3)$) {$c$};
		\node at ($(v)+(-.4,-.3)$) {$v$};
		\node at ($(w)+(.4,-.3)$) {$w$};
		\node at ($(u)+(0,.4)$) {$u$};
		\node at ($(x)+(.4,-.3)$) {$x$};
		\node at ($(y)+(.4,.3)$) {$y$};
		
		\node at ($(v1)+(.4,-.1)$) {\tiny $v_0$};
		\node at ($(v3)+(.1,-.35)$) {\tiny $v_{i-1}$};
		\node at ($(v4)+(-.1,-.3)$) {\tiny $v_{i}$};
		\node at ($(v6)+(0.05,-.35)$) {\tiny $v_{2i-1}$};
		\node at ($(v7)+(-.1,-.3)$) {\tiny $v_{2i}$};
		\node at ($(v8)+(0,.3)$) {\tiny $v_{3i-2}$};
		
		\end{pgfonlayer}
		
		\end{tikzpicture}
		\caption{Positive integers assigned to the vertices of the original graph $\Upsilon_i^{\mu\nu}$.}
		\label{fig:Vega_w}
	\end{figure}

One checks easily that the support of $\omega_{\mu\nu}$ always contains $V(\Upsilon_i^{11})$, 
that $y$ is in this support if and only if $\mu=0$, and that $v_{2i-1}$ is in this support if and only 
if $\nu=0$. Consequently, 
\begin{equation}\label{eq:0959}
	\text{ the support of $\omega_{\mu\nu}$ is $V(\Upsilon_i^{\mu\nu})$}\,.
\end{equation}

Given a signed weight function $h\colon V(\Upsilon_i)\longrightarrow \ZZ$ and a subset 
$A\subseteq V(\Upsilon_i)$ of the vertex set we define 
\[
	h(A)=\sum_{z\in A}h(z)\,.
\]
For instance, some quick calculations disclose the formulae
\begin{align*}
	\omega_{00}\bigl(V(\Upsilon_i)\bigr) &=27i-19 \\
	\text{ and } \quad f\bigl(V(\Upsilon_i)\bigr) &=g\bigl(V(\Upsilon_i)\bigr)=3\,,
\end{align*}
which by linearity,~\eqref{eq:1016}, and~\eqref{eq:0959} imply
\begin{equation}\label{eq:2212}
	\omega_{\mu\nu}\bigl(V(\Upsilon^{\mu\nu}_i)\bigr)=3\bigl(9i-(6+\mu+\nu)\bigr)-1\,.
\end{equation}
Similarly, for every vertex $z$ we have 
\begin{align*}
	\omega_{00}\bigl(\Nn(z)\bigr) &=9i-6\,, \\
	f\bigl(\Nn(z)\bigr) &=\begin{cases}
		1 & \text{ if $z\ne y$} \cr
		2 & \text{ if $z= y$,}
		\end{cases} \\
	\text {and} \quad
	g\bigl(\Nn(z)\bigr)&=\begin{cases}
		1 & \text{ if $z\ne v_{2i-1}$} \cr
		2 & \text{ if $z= v_{2i-1}$,}
		\end{cases}
\end{align*}
whence
\begin{equation}\label{eq:2227}
	\o\bigl(\Nn(z)\bigr)=9i-(6+\mu+\nu) 
	\quad \text{for all } z\in V(\Upsilon_i^{\mu\nu})\,.
\end{equation}

Let $G^{\mu\nu}_i = \ve(\o)$ be the blow-up of $\Upsilon^{\mu\nu}_i$ obtained by replacing 
every vertex $z$ by an independent set~$Z$ of size $\o(z)$, and let 
	\begin{equation}\label{eq:k}
		k=9i-(6+\mu+\nu)
	\end{equation}
be the number occurring in Lemma~\ref{l:vegaex}. We summarise~\eqref{eq:2212} 
and~\eqref{eq:2227} in the following observation due to~\cite{BP}.

\begin{fact}\label{f:G1}
	If $i\ge 2$ and $\mu,\nu \in \{0,1\}$, then $G^{\mu\nu}_i$ is a $k$-regular blow-up 
	of $\vv{i}{\mu\nu}$ on $3k-1$ vertices. \hfill \qed
\end{fact}

Finally, as in the case of Andr\'asfai graphs, we characterize all extremal Vega graphs.

\begin{dfn}\label{def:exvega}
	Given natural numbers $k\ge 10$ and $n\ge s\ge 1$ the family $\cH^n_k(s)$ is the smallest
	collection of blow-ups of Vega graphs with the following properties. 	
	\begin{enumerate}[label=\alabel]
		\item\label{it:1211a} If  $s = kn/(3k-1)$, then $G^{\mu\nu}_i(3s-n)\in \cH^n_k(s)$
			whenever $k=9i-(6+\mu+\nu)$.
		\item\label{it:1211b} If $kn/(3k-1) < s \le (k-1)n/(3k-4)$, then 
			$\Upsilon^{0\nu}_i\bigl((3s-n)\o-\lambda f\bigr)\in \cH^n_k(s)$ whenever $k=9i-(6+\nu)$.
		\item\label{it:1211c} If $kn/(3k-1) < s \le (k-1)n/(3k-4)$, then 
			$\Upsilon^{\mu0}_i\bigl((3s-n)\o-\lambda g\bigr)\in \cH^n_k(s)$ whenever $k=9i-(6+\mu)$.
	\end{enumerate}
\end{dfn}

Observe that for
\begin{itemize}
	\item $s\notin [kn/(3k-1) , (k-1)n/(3k-4)]$ the family $\cH^n_{k}(s)$ is empty;
	\item $k\nequiv 1,2,3\pmod{9}$ the family $\cH^n_{k}(s)$ is empty;
	\item $k\equiv 1\pmod{9}$ we have $\mu + \nu =2$ and the family $\cH^n_{k}(s)$ is nonempty if 
		and only if $s=nk/(3k-1)$, in which case it only consists of the 
		graph $G^{11}_{(k+8)/9}(3s-n)$;
	\item $k\equiv 2, 3\pmod{9}$ and $s=nk/(3k-1)$ there are one or two graphs 
		in $\cH^n_{k}(s)$ as described in~\ref{it:1211a}.
	\item $k\equiv 2, 3\pmod{9}$ and $kn/(3k-1) < s \le (k-1)n/(3k-4)$ the family 
		$\cH^n_{k}(s)$ consists of two graphs, one of which is as described in~\ref{it:1211b}  
		while the other one is as described in~\ref{it:1211c}.
\end{itemize}

\begin{lemma}\label{l:vegaextremal}
	Given integers $i\ge 2$, $\mu, \nu \in \{0,1\}$, and $n\ge s\ge 1$, let $H$ be a blow-up 
	of $\ve$ on $n$ vertices. If 
	\[
		\alpha(H)\le s
		\quad \text{ and } \quad 
		e(H)=g_k(n,s)
	\]
	hold for $k=9i-(6+\mu+\nu)$, then $H$ is isomorphic to a graph in the 
	family~$\cH^n_{k}(s)$.
\end{lemma}

\subsection{Proof of Lemma~\ref{l:vegaex}.}
Throughout this subsection, we fix some integers $i\ge 2$, $\mu, \nu\in\{0, 1\}$,
and $n\ge s\ge 1$. We keep using the weight function $\o\colon V(\vl)\lra\ZZ_{\ge 0}$
and the natural number $k$ defined in~\eqref{eq:1016} and~\eqref{eq:k}, respectively. 
Let $H$ be an $n$-vertex blow-up of $\Upsilon^{\mu\nu}_i$ such that $\alpha(H)\le s$.
It will be convenient to view $H$ as a blow-up of $\vl$ as well by adding an empty 
vertex class $Y$ in case $\mu=1$ and an empty $V_{2i-1}$ in case $\nu=1$. 
The independence of the neighbourhoods of the vertices in $\vl$ entails  
	\begin{equation}\label{eq:nh}
		|N(Z)|\le s
	\end{equation}
for every vertex class $Z\subseteq V(H)$ corresponding to a vertex $z\in \vl$
(even if this vertex $z$ fails to belong to $\ve$.)

We first bound the size of the vertex classes of $H$ from above and below. 
The ideas in the proofs of the two following facts are similar to those in 
the proof of Fact~\ref{Vi}. 

\begin{fact}\label{Vegaup}
	Every vertex $z\in V(\vl)$ with $\o (z)\ge 2$ satisfies		
	\[
		|Z|\le \o(z)(3s-n)\,.
	\]
	Moreover, we have 
	\begin{align*}	
		|X| + |Y| &\le 2(3s-n)=\o(\{x, y\})(3s-n) \\
		\text{ and } |V_0| + |V_{2i-1}| &\le 2(3s-n)=\o(\{v_0, v_{2i-1}\})(3s-n)\,.
	\end{align*}
\end{fact}

\begin{proof}
	The upper bound on $|X| +|Y|$ follows from  
	\begin{multline*}
		|X| +|Y| + 2n
		 = 
		|N(A)| + |N(V)| + |N(C)| + |N(U)| + |N(B)|+|N(W)| 
		\overset{\eqref{eq:nh}}{\le}6s\,.
	\end{multline*}
	Similarly,  
	\begin{multline*}
		|V_0| + |V_{2i-1}| + 2n
		 = 
		|N(A)| + |N(B)| + |N(U)| + |N(V)| + |N(V_{i-1})|+|N(V_i)| 
		\overset{\eqref{eq:nh}}{\le}
		6s
	\end{multline*}
	yields the desired upper bound on $|V_0| + |V_{2i-1}|$. It remains to prove 
	$|Z|\le \o(z)(3s-n)$ for every $z\in V(\vl)\sm\{v_0, v_{2i-1}, x, y\}$.
	Following the same strategy we have just used, this task reduces to exhibiting 
	a list of $3\o(z)$ vertices of $\Upsilon_i$ whose neighbourhoods cover the 
	entire vertex set $\o(z)$ many times and the set $Z$ itself even once more.
	As there are several cases and plenty of vertices to consider, it seems useful
	to devise the following notation. For a set $Q\subseteq V(\vl)$ we denote
	its characteristic function by $\Ind(Q)$, so
	\[
		\Ind(Q)(q)=
			\begin{cases}
				1 & \text{ if $q\in Q$} \cr
				0 & \text{ otherwise} \cr
			\end{cases}
	\]
	for all $q\in V(\vl)$.
	If $Q=\{q_1, \ldots, q_r\}$ comes with 
	an explicit enumeration of its members, it will be convenient to omit a pair of curly
	braces, thus writing $\Ind(q_1, \ldots, q_r)$ for this function. For instance, 
	the functions $f$ and~$g$ considered earlier can now be represented as 
	\[
		f=\Ind(u, v, w, y)-\Ind(x)
		\quad \text{ and } \quad 
		g=\Ind(b, v, v_{i-1}, v_{2i-1})-\Ind(v_0)\,.
	\]
	Instead of $\Ind\bigl(V(\vl)\bigr)$ we shall just write $\Ind$.
	Next, given a function $h\colon V(\vl)\lra\ZZ$ 
	we let $\Sigma(h)=h(V(\vl))=\sum_{t\in V(\vl)}h(t)$ 
	be the sum of the values it attains
	and by $\Nf(h)\colon V(\vl)\lra\ZZ$ we mean the function
	$\sum_{t\in V(\vl)}h(t)\,\Ind\bigl(\Nn(t)\bigr)$. In other words, this function satisfies 
	\begin{equation}\label{eq:1205}
		\Nf(h)(t)=\sum_{t'\in \Nn(t)} h(t') = h\bigl(\Nn(t)\bigr)
		\quad \text{ for every } t\in V(\vl)\,.
	\end{equation}
	So we used earlier that 
	\begin{align}
		\Nf(f)&=\Ind+\Ind(y)\,, & \Sigma(f)&=3\,, \notag \\
		\text{and } \quad \Nf(g)&=\Ind+\Ind(v_{2i-1})\,, &  \Sigma(g)&=3 \,, \label{eq:1051} 
	\end{align}
	the upper bound on $|X|+|Y|$ relies on the fact that the hexagon $\ccC_6=\{a,v,c,u,b,w\}$
	has the properties 
	\begin{equation}\label{eq:1043}
		\Nf\bigl(\Ind(\ccC_6)\bigr)=2\cdot \Ind+\Ind(x, y)\,,
		\quad
		\Sigma\bigl(\Ind(\ccC_6)\bigr)=|\ccC_6|=6
	\end{equation}
	and soon we are going to need that the inner Andr\'asfai graph 
	$\G=\{v_0,v_1,\dots, v_{3i-2}\}$ satisfies 
	\begin{equation}\label{eq:1203}
		\Nf\bigl(\Ind(\G)\bigr)=i\cdot \Ind(\G)+(i-1)\cdot \Ind(\ccC_6)+\Ind(a, b, u, v)\,,
		\qquad
		\Sigma\bigl(\Ind(\G)\bigr)=|\G|=3i-1\,.
	\end{equation}

	Now it suffices to exhibit for every vertex $z\in V(\vl)\sm\{v_0, v_{2i-1}, x, y\}$
	a function 
	\[
		h^z\colon V(\vl)\lra \ZZ_{\ge 0}
	\]
	such that 
	\begin{enumerate}[label=\alabel]
		\item\label{it:0111a} $\Nf(h^z)$ agrees on $V(\Upsilon^{\mu\nu}_i)$
			with $\o(z)\cdot \Ind+\Ind(z)$ 
		\item\label{it:0111b} and $\Sigma(h^z)=3\o(z)$.
	\end{enumerate}
	Indeed, once we have such a function $h^z$ at our disposal, we can imitate the 
	above reasoning and argue that 
	\begin{align*}
		\o(z)n+|Z|
		&\overset{\text{\ref{it:0111a}}}{=}
		\sum_{t\in V(\vl)}\Nf(h^z)(t) |T|
		\overset{\eqref{eq:1205}}{=} 
		\sum_{t\in V(\vl)} \sum_{t'\in \Nn(t)} h^z(t') \,|T| \\
		&= 
		\sum_{t'\in V(\vl)} h^z(t')\, |N(T')|
		\overset{\eqref{eq:nh}}{\le}  
		\Sigma(h^z) s
		\overset{\text{\ref{it:0111b}}}{=}
		3\o(z) s\,,
	\end{align*}
	which proves the desired inequality $|Z|\le \o(z)(3s-n)$.
	
	Starting with the vertices in $\G$ we observe that for every 
	$k\in \ZZ_{3i-1}\sm\{0, 2i-1\}$ we have 
	\[
		\Nf\bigl(\Ind(v_{k-i}, v_k, v_{k+i})\bigr)=\Ind(\ccC_6)+\Ind(\G)+\Ind(v_k)\,.
	\]
	By adding~\eqref{eq:1043} we infer 
	\begin{equation}\label{eq:1212}
		\Nf\bigl(\Ind(\ccC_6)+\Ind(v_{k-i}, v_k, v_{k+i})\bigr)=3\cdot\Ind+\Ind(v_k)\,,
	\end{equation}
	which shows that for $k\ne 0, i-1, 2i-1$ the function 
	$h^{v_k}=\Ind(\ccC_6)+\Ind(v_{k-i}, v_k, v_{k+i})$ has the required properties.
	Moreover, for $k=i-1$ we deduce from~\eqref{eq:1212} with the help 
	of~\eqref{eq:1051} that 
	\[
		\Nf\bigl(\Ind(\ccC_6)+\Ind(v_{i-1}, v_{2i-1}, v_{3i-2})-\nu g\bigr)
		=
		(3-\nu)\Ind+\Ind(v_{i-1})-\nu\Ind(v_{2i-1})\,,
	\]
	meaning that we can take $h^{v_{i-1}}=\Ind(\ccC_6)+\Ind(v_{i-1}, v_{2i-1}, v_{3i-2})
	-\nu g$. It remains to deal with the vertices on the hexagon. Starting with $a$, 
	we observe 
	\begin{equation}\label{eq:1246}
		\Nf\bigl(\Ind(a, x)-\Ind(v_{2i-1})\bigr)=\Ind(a, c, w, x, y)\,,
	\end{equation}
	which in view of~\eqref{eq:1043} and~\eqref{eq:1203} entails that the function
	\[
		h^a=(i-1)\Ind(\ccC_6) +\Ind(\Gamma)+\Ind(a, x)-\Ind(v_{2i-1})
	\]
	satisfies 
	\begin{multline*}
		\Nf(h^a)
		=[(2i-2)\cdot \Ind+(i-1)\cdot \Ind(x, y)] 
		+[i\cdot\Ind(\G)+(i-1)\cdot\Ind(\ccC_6) +\Ind(a, b, u, v)] \\
		+ \Ind(a, c, w, x, y) =(3i-2)\cdot\Ind+\Ind(a)\,.
	\end{multline*}
	Together with $\Sigma(h^a)=6(i-1)+(3i-1)+2-1=3(3i-2)$ this proves that $h^a$ 
	has all required properties. By symmetry, for 
	$\widetilde{h}^b=(i-1)\cdot\Ind(\ccC_6) +\Ind(\Gamma)+\Ind(b, x)-\Ind(v_{0})$
	we obtain $\Nf(\widetilde{h}^b)=(3i-2)\Ind+\Ind(b)$ and 
	$\Sigma(\widetilde{h}^b)=3(3i-2)$,
	so by~\eqref{eq:1051} we may set $h^b=\widetilde{h}^b-\nu g$.
	Next, 
	\begin{equation}\label{eq:1415}
		\Nf\bigl(\Ind(x)-\Ind(v_0, v_{2i-1}, w)\bigr)
		=
		\Ind(c)-\Ind(a, b, u, v)-\Ind(\G)
	\end{equation}
	reveals that the function 
	$h^c=(i-1)\cdot\Ind(\ccC_6)+\Ind(\G)+\Ind(x)-\Ind(v_0, v_{2i-1}, w)$
	has the property 
	\begin{multline*}
		\Nf(h^c)=[(2i-2)\cdot\Ind+(i-1)\cdot\Ind(x, y)]
			+[i\cdot\Ind(\G)+(i-1)\cdot\Ind(\ccC_6)+\Ind(a, b, u, v)] \\
			+[\Ind(c)-\Ind(a, b, u, v)-\Ind(\G)]  = (3i-3)\Ind+\Ind(c)
	\end{multline*}
	and because of $\Sigma(h^c)=6(i-1)+(3i-1)+1-3=3(3i-3)$ this establishes 
	the desired bound $|C|\le (3i-3)(3s-n)$. 
	Utilising that similar to~\eqref{eq:1246},~\eqref{eq:1415} we 
	have 
	\begin{align*}
		\Nf\bigl(\Ind(u, y)-\Ind(v_{2i-1})\bigr) &=\Ind(c, u, w, x, y) \\
		\Nf\bigl(\Ind(v, y)-\Ind(v_{0})\bigr) &= \Ind(c, v, w, x, y) \\
		\text{ and } \quad \Nf\bigl(\Ind(y)-\Ind(c, v_0, v_{2i-1})\bigr)
			&=\Ind(w)-\Ind(a, b, u, v)-\Ind(\G)
	\end{align*}
	one confirms easily that the functions 
	\begin{align*}
		h^u &=(i-1)\Ind(\ccC_6)+\Ind(\G)+\Ind(u, y)-\Ind(v_{2i-1})-\mu f \\
		h^v &=(i-1)\Ind(\ccC_6)+\Ind(\G) +\Ind(v, y)-\Ind(v_0)-\mu f-\nu g \\
		\text{ and } \quad 
		h^w &=(i-1)\Ind(\ccC_6)+\Ind(\G)+\Ind(y)-\Ind(c, v_0, v_{2i-1})-\mu f
	\end{align*}
	take care of the three remaining vertex classes.
\end{proof}

Adding all inequalities provided by Fact~\ref{Vegaup}
we obtain $n\le (3k-1)(3s-n)$ (recall~\eqref{eq:2212} and~\eqref{eq:k}), 
whence  
\[
	\frac{kn}{(3k-1)}\le s
\]
and the number  
\begin{equation}\label{eq:lamda}
	\lambda 
	=
	(3k-1)s-kn 
	= 
	k(3s-n) - s
\end{equation}
is nonnegative. For later use we rewrite this in the form 
\begin{equation}\label{eq:1709}
	n=(3k-1)(3s-n)-3\lambda
	\overset{\eqref{eq:2212}}{=} \o\bigl(V(\vl)\bigr)(3s-n)-3\lambda\,.
\end{equation}
Similarly,~\eqref{eq:nh} and~\eqref{eq:2227} yield 
\begin{equation}\label{eq:vs}
	|N(Z)| \le s = k(3s-n) - \lambda = \o\bigl(N(z)\bigr)(3s-n)-\lambda
\end{equation}
for every vertex $z\in V(\Upsilon^{\mu\nu}_i)$. Now we are ready for a lower bound on the sizes
of the vertex classes of $H$. 

\begin{fact}\label{Vegadown}
	If $z\in V(\vl)$, then
	\begin{equation}\label{vegalow}
			|Z|\ge \o(z)(3s-n)-\lambda\,.
	\end{equation}
\end{fact}
\begin{proof}
	As we shall prove by means of a complete case analysis, there exist two adjacent 
	non-neighbours $z'$, $z''$ of $z$ in $\Upsilon^{\mu\nu}_i$ such that each of the sets 
	$\{x, y\}$ and $\{v_0, v_{2i-1}\}$ is either contained in
	$\Nn(z')\cup \Nn(z'')\cup \{z\}$ or in its complement. 
	
	Once we have two such vertices $z'$ and $z''$, the argument proceeds as follows. 
	Due to $z'z''\in E(\vl)$ and $zz', zz''\not\in E(\vl)$ the union 
	$R=\{z\}\dcup \Nn(z')\dcup \Nn(z'')$ is disjoint. 
	Let $S=V(\vl)\sm R$ be 
	the complement of $R$ and write 
	\[
		R_H=Z\dcup N(Z')\dcup N(Z'')
		\quad \text{ as well as } \quad
		S_H=V(H)\sm R_H
	\]
	for the sets corresponding to $R$ and $S$ in $H$. Fact~\ref{Vegaup} implies 
	\begin{equation}\label{eq:1831}
		|S_H|\le \o(S)(3s-n)
	\end{equation}
	and, consequently, we have 
	\begin{align*}
		|Z| &+|N(Z')|+|N(Z'')|
		=
		|R_H|
		=
		n-|S_H| \\
		&\overset{\eqref{eq:1709}}{\ge} 
		\o\bigl(V(\vl)\bigr)(3s-n)-3\lambda-\o(S)(3s-n) 
		= 
		\o(R)(3s-n)-3\lambda \\
		&\overset{\eqref{eq:vs}}{\ge} 
		\o(z)(3s-n)-\lambda+|N(Z')|+|N(Z'')|\,,
	\end{align*}
	i.e., $|Z|\ge \o(z)(3s-n)-\lambda$. So it remains to check that the auxiliary vertices 
	$z'$ and $z''$ do indeed exist and we list some possible choices in the table that
	follows.
	
	\begin{center}
	\begin{tabular}{|c|c|c|c|c|c|c|c|c|}
	\hline 
	$z$&  $v_0$&  $v_1, \ldots, v_{2i-2}$&
	$v_{2i-1}$&  $v_{2i}, \ldots, v_{3i-2}, a, v$&  $b, u$&
	$c, w$&  $x$&  $y$ \\ \hline
 	$z'$, $z''$&  $c$, $v$&  $c$, $x$&
	$c$, $u$&  $b$, $u$&  $a$, $v$&
	$v_0$, $v_i$&  $u$, $v_0$&  $a$, $v_0$ \\ \hline 
\end{tabular}
\end{center}
	This concludes the proof of Fact~\ref{Vegadown}.
\end{proof}

We are left with the task of proving Lemma~\ref{l:vegaex} itself. To this end we remark 
that~$H$ can be regarded as a blow-up of the graph $G^{\mu\nu}_i$ considered in 
Fact~\ref{f:G1}. In fact, every vertex~$z$ of~$\Upsilon^{\mu\nu}_i$ corresponds to 
some vertex class $Z$ of $H$ and due to $\o(z)\ge 1$ we can partition each such vertex 
class into $\o(z)$ ``particles'' of size $\lfloor|Z|/\o(z)\rfloor$ or 
$\lceil|Z|/\o(z)\rceil$ each. Owing to Fact~\ref{Vegadown} the sizes of these particles 
are at least 
\begin{align}\label{eq:226}
	\left\lfloor\frac{|Z|}{\o(z)}\right\rfloor
	&\overset{\eqref{vegalow}}{\ge}
	\left\lfloor\frac{\o(z)(3s-n)-\lambda}{\o(z)}\right\rfloor
	=
	3s-n-\left\lceil\frac{\lambda}{\o(z)}\right\rceil \\
	&\ge 3s-n-\lambda
	\overset{\eqref{eq:lamda}}{=}
	(k-1)n-(3k-4)s\,.\nonumber
\end{align}
As the particles endow $H$ with the structure of a blow-up of $G^{\mu\nu}_i$,
Corollary~\ref{blowupedges} shows that we have indeed
\[
	e(H) \le g_k(n,s)\,.
\]
This concludes the proof of Lemma~\ref{l:vegaex}. \hfill \qed

\subsection{Proof of Lemma \ref{l:vegaextremal}.}

Corollary~\ref{blowupedges}\ref{it:2140a} yields
\[
	\frac{kn}{3k-1} \le s \le \frac{(k-1)n}{3k-4}\,,
\]
whence $0 \le \lambda \le 3s-n$. In the special case $\lambda=0$ the Facts~\ref{Vegaup} 
and~\ref{Vegadown} imply that every vertex $z\in V(\ve)$ satisfies $|Z|=\o(z)(3s-n)$, 
meaning that statement~\ref{it:1211a} holds. From now on we suppose $\lambda\ge 1$, 
which entails the estimate on $s$ occurring in~\ref{it:1211b} and~\ref{it:1211c}.

\begin{claim}\label{clm:1315}
	There exists a vertex $z\in V(\ve)\cap\{x, y, v_0, v_{2i-1}\}$ such that $\o(z)=1$ and 
	$|Z| = 3s-n-\lambda$. 
\end{claim}

\begin{proof}
	Recall that in the proof of Lemma~\ref{l:vegaex} we viewed $H$ as a blow-up of 
	the $k$-regular graph $G_i^{\mu\nu}$. By Corollary~\ref{blowupedges} at least
	one of the particles occurring in this construction has the 
	size $(k-1)n-(3k-4)s = (3s-n) - \lambda$. Let $z\in V(\ve)$ be a vertex one of 
	whose~$\o(z)$ particles has this size. Now~\eqref{eq:226} needs to hold with 
	equality and we have $\lceil \lambda/\o(z)\rceil=\lambda$. For this reason 
	at least one of the equations $\o(z)=1$ or $\lambda=1$ is true.  
	In the former case, $z\in \{x, y, v_0, v_{2i-1}\}$ is immediate and we are done. 
	
	Now suppose for the sake of contradiction that $\lambda=1$ 
	and that $|Z|\ne 3s-n-\lambda$ holds for every $z\in V(\ve)\cap\{x, y, v_0, v_{2i-1}\}$ 
	with $\o(z)=1$. Together with the Facts~\ref{Vegaup} and~\ref{Vegadown} this 
	yields $|Z|\le \o(z)(3s-n)$ for all vertices $z\in V(\ve)$. 
	
	Concerning the set $M=\{z\in V(\ve)\colon |Z| < \o(z)(3s-n)\}$ of vertices for which 
	this estimate fails to be sharp we can deduce from 
	\[
		 n=\sum_{z\in V(\ve)} |Z| \le \o\bigl(V(\ve)\bigr)(3s-n) -|M|
	\]
	and~\eqref{eq:1709} that $|M|\le 3$. Owing to $\chi(\ve)=4$ this shows that the neighbourhoods
	of the vertices in $M$ cannot cover the entire vertex sets of $\ve$ or, in other words, that 
	there is a vertex $z_\star\in V(\ve)$ whose neighbourhood is disjoint to $M$.	
	But now 
	\[
		s
		\ge 
		|N(Z_\star)| 
		= 
		\sum_{z\in \Nn(z_\star)}|Z|
		= 
		\sum_{z\in \Nn(z_\star)}\o(z)(3s-n)
		=
		k(3s-n)
	\]
	contradicts~\eqref{eq:vs}. This completes the proof of Claim~\ref{clm:1315}. 
\end{proof}

Let us observe that if the vertex $z$ delivered by Claim~\ref{clm:1315} is either $x$ or $y$, 
then $\mu=0$, while if it is one of $v_0$, $v_{2i-1}$, then $\nu=0$. 
	
\smallskip

{\it \hskip2em First Case. We have $\mu=0$ and one of $X$, $Y$ has size $3s-n-\lambda$.}
	
\medskip

The product of the four transpositions $x\longleftrightarrow y$, $a\longleftrightarrow u$,
$b\longleftrightarrow v$, and $c\longleftrightarrow w$ is an automorphism of $\ve$ and, 
since we only aim at determining $H$ up to isomorphism this fact shows that without loss 
of generality we may suppose $|Y|=3s-n-\lambda$.

Now for $y$ in place of $z$ the proof of Fact~\ref{Vegadown} goes through with equality.
In particular, if we work with the pair $\{z', z''\}=\{a, v_0\}$ indicated in the table,
we need to have equality in~\eqref{eq:1831}, which in turn implies in view of Fact~\ref{Vegaup}
that 
\begin{equation}\label{eq:equal}
	|Z|=\o(z)(3s-n)
\end{equation}
holds for all $z\in \{b, c, v_{2i}, \ldots, v_{3i-2}\}$. Working with the 
pair $\{a, v_{i-1}\}$ or $\{b, v_i\}$ instead we learn that~\eqref{eq:equal}
is also valid for all $z\in \{v_i, \ldots, v_{2i-2}\}$ and 
all $z\in \{a, v_1, \ldots, v_{i-1}\}$. Altogether, this proves~\eqref{eq:equal}
for all $z\ne v_0, v_{2i-1}, u, v, w, x, y$.

Now $|N(V_{2i})|\le s=\o\bigl(\Nn(v_{2i})\bigr)(3s-n)-\lambda$ and Fact~\ref{Vegadown}
yield 
\[
	|W|=\o(w)(3s-n)-\lambda\,, 
\]
meaning for $z=w$ the estimates entering the 
proof of Fact~\ref{Vegadown} hold with equality as well. Applying this observation  
to $\{v_0, v_i\}$ playing the r\^ole of $\{z', z''\}$ and to~\eqref{eq:1831} we 
conclude $|X|+|Y|=2(3s-n)$, which in turn entails $|X|=(3s-n)+\lambda$.
Thus $|N(C)|\le s$ and Fact~\ref{Vegadown} are only compatible if $|Z|=\o(z)(3s-n)-\lambda$
holds for $z=u, v$ as well.
 
Finally, $|V_i|>\o(v_i)(3s-n)-\lambda$ and Corollary~\ref{blowupedges}\ref{it:2140c} yield 
$|N(V_i)|=s$, whence 
\[
	|V_0|=\o(v_0)(3s-n)\,. 
\]
The same argument applied to $V_{i-1}$
discloses $|V_{2i-1}|=\o(v_{2i-1})(3s-n)$ and altogether this concludes the proof
that $|Z|=\o(z)(3s-n)-\lambda f(z)$ holds for every $z\in V(\ve)$, i.e., 
that $H$ is as described in~\ref{it:1211b}.
 
\smallskip

{\it \hskip2em Second Case. We have $\nu=0$ and one of $V_0$, $V_{2i-1}$ 
has size $3s-n-\lambda$.}
	
\medskip
	
We will show that outcome~\ref{it:1211c} of our lemma occurs. The argument will 
be very similar to the one we saw in previous case. First, we note that
the reflection $v_j\longmapsto v_{2i-1-j}$ of $\Gamma$ composed with the transpositions 
$a\longleftrightarrow b$, $u\longleftrightarrow v$ constitutes an automorphism of~$\ve$
that exchanges $v_0$ and $v_{2i-1}$. Thus we may suppose without loss of generality 
that $|V_{2i-1}|=3s-n-\lambda$. 

As before, we need to have equality in~\eqref{eq:1831} when running the proof of
Fact~\ref{Vegadown} for $z=v_{2i-1}$ and $\{z', z''\}$ being one of the pairs $\{a, x\}$, 
$\{a, w\}$, or $\{c, u\}$. Consequently,~\eqref{eq:equal} holds whenever 
$z\in \{v_i, \ldots, v_{2i-2}, v_{2i}, \ldots, v_{3i-2}, a, c, u, w\}$. 
Now $|N(V_{3i-2})|\le s$ and Fact~\ref{Vegadown} yield 
\[
	|V_{i-1}|=\o(v_{i-1})(3s-n)-\lambda\,.
\]

This equality case of Fact~\ref{Vegadown} can be analysed by using the 
pair $\{c, x\}$ in place of $\{z', z''\}$. 
In this manner we infer that~\eqref{eq:equal} holds for 
$z\in\{v_1, \ldots, v_{i-2}\}$ as well and, moreover, that 
$|V_0|=(3s-n)+\lambda$. Together with $|N(V_i)|\le s$ and~Fact~\ref{Vegadown}
this establishes $|Z|=\o(z)(3s-n)-\lambda$ for $z=b, v$ and it remains to 
check~\eqref{eq:equal} for $z=x, y$. To this end, we observe that $|A|>\o(a)(3s-n)-\lambda$ 
and Corollary~\ref{blowupedges}\ref{it:2140c} yield $|N(A)|=s$, whence $|X|=\o(x)(3s-n)$. 
The argument for $Y$ is similar but considers $U$ instead of~$A$. \qed

\section{Proofs of the main results}
\label{sec:2141}

The main ingredient of our argument is a result of Brandt and Thomass\'e~\cite{BT} 
which states that all maximal triangle-free graphs of large minimum degree are blow-ups 
of Andr\'asfai and Vega graphs. A graph $G$ is \emph{maximal triangle-free} if adding 
any missing edge to $G$ creates a triangle.

\begin{theorem}\label{th:dense}
	Every maximal triangle-free graph on $n$ vertices with minimum degree greater 
	than $n/3$ is a proper blow-up of some Andr\'asfai graph or Vega graph. 
\end{theorem}

Next we recall~\cite{LPR}*{Fact 2.6}, which will allows us to restrict to the 
`correct' range of $s$ when proving Theorem~\ref{thm:mindeg}.

\begin{fact}\label{f:s}
	If $n\ge s \ge 1$ and $k\ge 2$ are such that 
	${s\notin \bigl(\frac{k}{3k-1}n, \frac{k-1}{3k-4}n\bigr)}$, then 
	\[
	\ex(n,s)\le g_k(n, s)
	\]
\end{fact}

We are now ready for the proof of our second main result. 

\begin{proof}[Proof of Theorem~\ref{thm:mindeg}]
	Observe that adding any edges to $H$ can neither increase its independence number 
	nor decrease its minimum degree. Thus, we may and shall assume that~$H$ is a maximal
	triangle-free graph.
	
	Due to Fact~\ref{f:s}, it suffices to consider the case
	\[
		\Delta(H) \le \alpha(H) \le s < \frac{k-1}{3k-4}n\,.
	\]
	Since $\delta(H) > (k+1)n/(3k+2) > n/3$, Theorem~\ref{th:dense} tells us 
	that $H$ is a proper blow-up of some graph $J$, which is either
	an Andr\'asfai graph $\Gamma_{\ell}$, or a Vega graph $\ve$. 
	In the latter case we set $\ell=9i-(6+\mu+\nu)$. Now in both cases $J$ has 
	a proper $\ell$-regular blow-up on $3\ell-1$ vertices and Lemma \ref{l:kregular} 
	yields $k = \ell$. 
	If $J=\Gamma_k$ is an Andr\'asfai graph the assertion follows from Lemma~\ref{ag0} 
	and if $J$ is a Vega graph we use Lemma~\ref{l:vegaex} instead.
\end{proof}

The other main result follows by means of a vertex deletion argument.
 
\begin{proof}[Proof of Theorem~\ref{thm:main}]
	Define   
	\[
		\zeta=\frac{1}{8k^2}
		\quad\text{ and } \quad
		\gamma=\frac{1}{600k^6}\,.
	\]
	Consider a triangle-free graph $G$ on $n$ vertices with $\alpha(G)\le s$ 
	and $e(G)=\ex(n,s)$, where 
	\begin{equation}\label{eq:1704}
		\frac{k}{3k-1}n \le  s \le \left(\frac{k}{3k-1}+\gamma\right)n\,.
	\end{equation}

	Since $\gamma$ is sufficiently small, we have $s<(k-1)n/(3k-4)$ and 
	Fact~\ref{f:s} implies $e(G)\le g_\ell(n, s)$ for every $\ell\ne k$. 
	On the other hand, Lemma~\ref{lem:2130} yields
	\begin{equation}\label{eq:0007}
		e(G)=\ex(n,s)\ge g_k(n,s)
	\end{equation}
	and it remains to prove that this holds with equality. It will be convenient to 
	rewrite~\eqref{eq:0007} as
	\[
		e(G)
		\ge 
		g_k(n,s)
		=
		\tfrac12\bigl[ns-\bigl((k-1)n -(3k-4)s\bigr)\bigl((3k-1)s - kn\bigr)\bigr]\,.
	\]
	Since
	\[
		0 \le (3k-1)s-kn < 3k\gamma n
	\]
	and 
	\[
		(k-1)n-(3k-4)s \le \frac{(3k-1)(k-1)}k s-(3k-4)s =\frac sk\,,
	\]
	we have 
	\[
		\bigl((k-1)n-(3k-4)s\bigr) \bigl((3k-1)s-kn\bigr)<3\gamma ns\,,
	\]
	and thus
	\[
	ns-2e(G) < 3\gamma ns\,.
	\]

	Therefore, the set 
	\begin{equation}\label{eq:Adef}
		A=\left\{v\in V(G)\colon d(v) < \left(1-\zeta\right)s\right\}
	\end{equation}
	satisfies 
	\[
	\zeta s|A|\le \sum_{v\in V}\left(s-d(v)\right)=ns-2e(G) < 3\gamma ns
	\]
	and, consequently, 
	\begin{equation}\label{eq:0027}
	|A| < \frac{3\gamma n}\zeta =\frac n{25k^4}\,.
	\end{equation}

	Now our argument will proceed as follows. We shall verify 	
	that the graph $G'=G-A$ satisfies the assumptions of Theorem~\ref{thm:mindeg} 
	and, hence, $e(G')$ is bounded from above by $g_k(n-|A|,s)$. Then, using the 
	fact that all vertices in $A$ are of small degree, we derive an upper bound 
	of $g_k(n,s)$ on the number of edges in $G$.
	
	For the minimum degree of $G'$ we obtain 
	\begin{align*}
		\delta(G') &\ge \left(1-\zeta\right)s-|A| 
		> \frac{k}{3k-1}n-\frac{n/2}{8k^2}-\frac{n}{25k^4}\\	
		&>\left(\frac{k+1}{3k+2}+\frac{1}{12k^2}-\frac 1{16k^2}-\frac 1{100k^2}\right)n \\
		&> \frac{k+1}{3k+2}n \ge \frac{k+1}{3k+2} |V(G')|\,.   
	\end{align*}
	As the graph $G'$ is triangle-free and satisfies $\alpha(G')\le s$, this shows that 
	Theorem~\ref{thm:mindeg} applies to~$G'$ and we are led to 
	\begin{equation}\label{eq:0028}
		e(G')\le g_k(n-|A|,s)
		=
		g_k(n, s)-|A|\left(k(k-1)n-\tfrac 12k(k-1)|A|-k(3k-4)s\right)\,.
	\end{equation}
	Now our choice of $\zeta$ and $\gamma$ yields 
	\begin{align*}
		\bigl((k-1)(3k-1)-\zeta\bigr)s
		&\le 
		\left((k-1)(3k-1)-\frac 1{8k^2}\right)\left(\frac k{3k-1}+\frac 1{600k^6}\right)n \\
		&<
		\left(k(k-1) +\frac 1{200k^4}-\frac 1{8k(3k-1)}\right)n \\
		&<
		\left(k(k-1) +\frac 1{50k^2}-\frac 1{25k^2}\right)n
		<k(k-1)\left(1-\frac 1{50k^4}\right)n \\
		&
		\overset{\eqref{eq:0027}}{<}
		k(k-1)(n-\tfrac 12 |A|)\,,
	\end{align*}
	and for this reason~\eqref{eq:0028} can be continued to 
	\begin{align*}
		e(G')&\le g_k(n, s)-\bigl((k-1)(3k-1)-\zeta-k(3k-4)\bigr)|A|s \\ 
		&= g_k(n,s) - (1-\zeta)|A|s\,.
	\end{align*}
	Every vertex in $A$ has degree at most $(1-\zeta)s$ in $G$, so we arrive at
	\begin{equation}\label{eq:0103}
		e(G)\le e(G')+(1-\zeta)|A|s\le  g_k(n, s)\,,
	\end{equation}
	which together with~\eqref{eq:0007} concludes the proof of Theorem~\ref{thm:main}.	
\end{proof}

Finally, let us remark that our results allow to determine the extremal graphs for 
Theorem~\ref{thm:main}.

\begin{cor}
	Suppose that $k\ge 2$ and that $G$ denotes a triangle-free graph on $n$ 
	vertices with $\alpha(G)\le s$ for some integer $s$ satisfying
	\[
		\frac{k}{3k-1}n \le  s \le \left(\frac{k}{3k-1}+\frac 1{600k^6}\right)n\,.
	\]
	If $e(G)=\ex(n, s)$, then $G$ is isomorphic to a graph in $\cG^n_k(s)\cup\cH^n_k(s)$.
\end{cor}

\begin{proof}
	Following the proof of Theorem~\ref{thm:main}, we see that~\eqref{eq:0103} 
	holds with equality, for which reason the set $A$ defined in~\eqref{eq:Adef}
	has to be empty. In other words, $G'=G$ and the proof of Theorem~\ref{thm:mindeg}
	discloses that $G'$ is a blow-up of either the Andr\'asfai graph $\Gamma_k$, or of 
   a Vega graph $\Upsilon^{\mu\nu}_i$ with $k=9i-(6+\mu+\nu)$. In the former case, 
   Lemma~\ref{lem:2130} shows that $G$ is isomorphic to a graph in $\cG^n_k(s)$
   and in the latter case we apply Lemma~\ref{l:vegaextremal}.  
\end{proof}

In fact, we strongly suspect that these are the only extremal graphs for the whole 
range of $s$, i.e., that the following stronger version of our initial conjecture holds.

\begin{conj}\label{conj:1859}
	If  $n/3<s\le n/2$, then 
	\begin{equation*}
		\ex(n,s)= \min_k g_k(n,s)\,,
	\end{equation*}
	where $g_k(n,s)$ is defined by~\eqref{eq2}. 
	Moreover, each extremal graph with $\ex(n,s)$ edges is isomorphic to one of the graphs 
	from the families~$\cG^n_k(s)$ and~$\cH^n_k(s)$ described in the Definitions~\ref{def:exan} 
	and~\ref{def:exvega}, respectively. 
\end{conj}


\end{document}